\documentclass[a4paper,oneside,english]{amsart}
\usepackage[T1]{fontenc}
\usepackage[utf8]{inputenc}
\usepackage{mathtools}
\usepackage{amstext}
\usepackage{amsthm}
\usepackage{amssymb}
\PassOptionsToPackage{normalem}{ulem}
\usepackage{ulem}
\usepackage[language=english, backend=biber, style=numeric, bibstyle=bibstyle, sorting=nyt, minnames=3, maxnames=99, autolang=hyphen, backref=true, hyperref=true, doi=false, isbn=false, url=false]{biblatex}

\makeatletter
\pdfpageheight\paperheight
\pdfpagewidth\paperwidth

\numberwithin{equation}{section}
\numberwithin{figure}{section}
  \theoremstyle{plain}
  \newtheorem*{thm*}{\protect\theoremname}
\theoremstyle{plain}
\newtheorem{thm}{\protect\theoremname}[section]
  \theoremstyle{definition}
  \newtheorem{example}[thm]{\protect\examplename}
  \theoremstyle{definition}
  \newtheorem{defn}[thm]{\protect\definitionname}
  \theoremstyle{remark}
  \newtheorem{rem}[thm]{\protect\remarkname}
  \theoremstyle{plain}
  \newtheorem{lem}[thm]{\protect\lemmaname}
  \theoremstyle{plain}
  \newtheorem{prop}[thm]{\protect\propositionname}
  \theoremstyle{plain}
  \newtheorem{cor}[thm]{\protect\corollaryname}

\usepackage[centering]{geometry}

\usepackage{etoolbox}
\usepackage{ifthen}
\usepackage{tikz}
\usetikzlibrary{calc}
\usepackage{tikz-cd}
\usetikzlibrary{arrows, matrix}

\tikzcdset{%
    triple line/.code={\tikzset{%
        double distance = 3pt, 
        double=\pgfkeysvalueof{/tikz/commutative diagrams/background color}}},
    quadruple line/.code={\tikzset{%
        double distance = 5.3pt, 
        double=\pgfkeysvalueof{/tikz/commutative diagrams/background color}}},
    Rrightarrow/.code={\tikzcdset{triple line}\pgfsetarrows{tikzcd implies cap-tikzcd implies}},
    RRightarrow/.code={\tikzcdset{quadruple line}\pgfsetarrows{tikzcd implies cap-tikzcd implies}}
}
\newcommand*{\tarrow}[2][]{\arrow[Rrightarrow, #1]{#2}\arrow[dash, shorten >= 0.5pt, #1]{#2}}

\makeatletter
\tikzset{
  column sep/.code=\def\pgfmatrixcolumnsep{\pgf@matrix@xscale*(#1)},
  row sep/.code   =\def\pgfmatrixrowsep{\pgf@matrix@yscale*(#1)},
  matrix xscale/.code=%
    \pgfmathsetmacro\pgf@matrix@xscale{\pgf@matrix@xscale*(#1)},
  matrix yscale/.code=%
    \pgfmathsetmacro\pgf@matrix@yscale{\pgf@matrix@yscale*(#1)},
  matrix scale/.style={/tikz/matrix xscale={#1},/tikz/matrix yscale={#1}}}
\def\pgf@matrix@xscale{1}
\def\pgf@matrix@yscale{1}
\makeatother

\tikzcdset{every label/.append style = {font = \normalsize}}

\makeatletter
\@namedef{subjclassname@2020}{%
  \textup{2020} Mathematics Subject Classification}
\makeatother

\numberwithin{thm}{subsection}

\usepackage{hyperref}

\makeatother

\usepackage{babel}
  \providecommand{\corollaryname}{Corollary}
  \providecommand{\definitionname}{Definition}
  \providecommand{\examplename}{Example}
  \providecommand{\lemmaname}{Lemma}
  \providecommand{\propositionname}{Proposition}
  \providecommand{\remarkname}{Remark}
  \providecommand{\theoremname}{Theorem}
\providecommand{\theoremname}{Theorem}

\begin{document}
\global\long\def\Div{\operatorname{Div}}

\global\long\def\ev{\operatorname{ev}}

\global\long\def\EV{\operatorname{EV}}

\global\long\def\NF{\operatorname{NF}}

\global\long\def\nf{\operatorname{nf}}

\global\long\def\Inv{\operatorname{Inv}}

\global\long\def\Gar{\operatorname{Gar}}

\global\long\def\im{\operatorname{im}}

\global\long\def\End{\operatorname{End}}

\global\long\def\sh{\operatorname{sh}}

\global\long\def\h{\operatorname{h}}

\global\long\def\t{\operatorname{t}}

\global\long\def\Std{\operatorname{Std}}

\newrobustcmd{\ifthen}[2]{\ifthenelse{#1}{#2}{}}
 
\newrobustcmd{\typedeuxbase}[3]{
\!\begin{tikzpicture}[
              scale = 0.5, every node/.style = { inner sep = 0mm, outer sep = 0mm },
              onecell/.style = { text height = 1.5ex, text depth = 0ex },
              ]
             
\node [onecell] (u) at (0,0) {#1} ;
\node [onecell] (v) at (1,0) {#2} ;
 
\ifthen{\equal{#3}{1}}
              {
                            \node (uv) at ($ (u.north) + (0.5,0.2) $) {} ;
                            \draw ($ (u.north) + (0.1,-0.1) $) -- (uv.south) -- ($ (v.north) + (-0.1,-0.1) $) ;
              }
\ifthen{\equal{#3}{0}}
              {
                            \node (uv) at ($ (u.north) + (0.5,0) $) {$\scriptstyle\times$} ;
              }
 
\end{tikzpicture}\!
}
 
\newrobustcmd{\typedeux}[1]{\typedeuxbase{$u$}{$v$}{#1}}
 
\newrobustcmd{\typetroisbase}[4]{
\!\begin{tikzpicture}[
              scale = 0.5, every node/.style = { inner sep = 0mm, outer sep = 0mm },
              onecell/.style = { text height = 1.5ex, text depth = 0ex },
              ]
 
\node [onecell] (u) at (0,0) {#1} ;
\node [onecell] (v) at (1,0) {#2} ;
\node [onecell] (w) at (2,0) {#3} ;
 
\node (uv) at ($ (u.north) + (0.5,0.2) $) {} ;
              \draw ($ (u.north) + (0.1,-0.1) $) -- (uv.south) -- ($ (v.north) + (-0.1,-0.1) $) ;
\node (vw) at ($ (v.north) + (0.5,0.2) $) {} ;
              \draw ($ (v.north) + (0.1,-0.1) $) -- (vw.south) -- ($ (w.north) + (-0.1,-0.1) $) ;
 
\ifthen{\equal{#4}{1}}
              {
                            \node (uvw) at ($ (uv.north) + (0.5,0.3) $) {} ;
                            \draw ($ (uv.north) + (0.1,0) $) -- (uvw.south) -- ($ (vw.north) + (-0.1,0) $) ;
              }
\ifthen{\equal{#4}{0}}
              {\node (uvw) at ($ (uv.north) + (0.5,0) $) {$\scriptstyle\times$} ;}
 
\end{tikzpicture}\!
}
 
\newrobustcmd{\typetrois}[1]{\typetroisbase{$u$}{$v$}{$w$}{#1}}
 
\newcommand{\typequatrebase}[7]{
\!\begin{tikzpicture}[
scale = 0.5, every node/.style = { inner sep = 0mm, outer sep = 0mm },
onecell/.style = { text height = 1.5ex, text depth = 0ex },
]
 
\node [onecell] (u) at (0,0) {#1} ;
\node [onecell] (v) at (1,0) {#2} ;
\node [onecell] (w) at (2,0) {#3} ;
\node [onecell] (x) at (3,0) {#4} ;
 
\node (uv) at ($ (u.north) + (0.5,0.2) $) {} ;
\draw ($ (u.north) + (0.1,-0.1) $) -- (uv.south) -- ($ (v.north) + (-0.1,-0.1) $) ;
\node (vw) at ($ (v.north) + (0.5,0.2) $) {} ;
\draw ($ (v.north) + (0.1,-0.1) $) -- (vw.south) -- ($ (w.north) + (-0.1,-0.1) $) ;
\node (wx) at ($ (w.north) + (0.5,0.2) $) {} ;
\draw ($ (w.north) + (0.1,-0.1) $) -- (wx.south) -- ($ (x.north) + (-0.1,-0.1) $) ;
 
\ifthen{\equal{#5}{1}}
{
\node (uvw) at ($ (uv.north) + (0.5,0.3) $) {} ;
\draw ($ (uv.north) + (0.1,0) $) -- (uvw.south) -- ($ (vw.north) + (-0.1,0) $) ;
}
\ifthen{\equal{#5}{0}}
{\node[inner sep=1.75pt] (uvw) at ($ (uv.north) + (0.5,0) $) {$\scriptstyle\times$} ;}
\ifthen{\equal{#6}{1}}
{
\node (vwx) at ($ (vw.north) + (0.5,0.3) $) {} ;
\draw ($ (vw.north) + (0.1,0) $) -- (vwx.south) -- ($ (wx.north) + (-0.1,0) $) ;
}
\ifthen{\equal{#6}{0}}
{\node[inner sep=1.75pt] (vwx) at ($ (vw.north) + (0.5,0) $) {$\scriptstyle\times$} ;}
 
%\ifthen{\equal{#5}{1}}
% {\ifthen{\equal{#6}{1}}
% {
\ifthen{\equal{#7}{1}}
{
\node (uvwx) at ($ (uvw.north) + (0.5,0.3) $) {} ;
\draw ($ (uvw.north) + (0.1,0) $) -- (uvwx.south) -- ($ (vwx.north) + (-0.1,0) $) ;
}
\ifthen{\equal{#7}{0}}
{\node (uvwx) at ($ (uvw.north) + (0.5,0) $) {$\scriptstyle\times$} ;}
% }
% }
\end{tikzpicture}\!
}
 
\newcommand{\typequatre}[3]{\typequatrebase{$u$}{$v$}{$w$}{$x$}{#1}{#2}{#3}}
 
\newcommand{\typecinq}[6]{
\!\begin{tikzpicture}[
scale = 0.5, every node/.style = { inner sep = 0mm, outer sep = 0mm },
onecell/.style = { text height = 1.5ex, text depth = 0ex },
]
\node [onecell] (u) at (0,0) {$u$} ;
\node [onecell] (v) at (1,0) {$v$} ;
\node [onecell] (w) at (2,0) {$w$} ;
\node [onecell] (x) at (3,0) {$x$} ;
\node [onecell] (y) at (4,0) {$y$} ;
 
\node (uv) at ($ (u.north) + (0.5,0.2) $) {} ;
\draw ($ (u.north) + (0.1,-0.1) $) -- (uv.south) -- ($ (v.north) + (-0.1,-0.1) $) ;
\node (vw) at ($ (v.north) + (0.5,0.2) $) {} ;
\draw ($ (v.north) + (0.1,-0.1) $) -- (vw.south) -- ($ (w.north) + (-0.1,-0.1) $) ;
\node (wx) at ($ (w.north) + (0.5,0.2) $) {} ;
\draw ($ (w.north) + (0.1,-0.1) $) -- (wx.south) -- ($ (x.north) + (-0.1,-0.1) $) ;
\node (xy) at ($ (x.north) + (0.5,0.2) $) {} ;
\draw ($ (x.north) + (0.1,-0.1) $) -- (xy.south) -- ($ (y.north) + (-0.1,-0.1) $) ;
 
\ifthen{\equal{#1}{1}}
{
\node (uvw) at ($ (uv.north) + (0.5,0.3) $) {} ;
\draw ($ (uv.north) + (0.1,0) $) -- (uvw.south) -- ($ (vw.north) + (-0.1,0) $) ;
}
\ifthen{\equal{#1}{0}}
{\node[inner sep=1.75pt] (uvw) at ($ (uv.north) + (0.5,0) $) {$\scriptstyle\times$} ;}
\ifthen{\equal{#2}{1}}
{
\node (vwx) at ($ (vw.north) + (0.5,0.3) $) {} ;
\draw ($ (vw.north) + (0.1,0) $) -- (vwx.south) -- ($ (wx.north) + (-0.1,0) $) ;
}
\ifthen{\equal{#2}{0}}
{\node[inner sep=1.75pt] (vwx) at ($ (vw.north) + (0.5,0) $) {$\scriptstyle\times$} ;}
 
\ifthen{\equal{#3}{1}}
{
\node (wxy) at ($ (wx.north) + (0.5,0.3) $) {} ;
\draw ($ (wx.north) + (0.1,0) $) -- (wxy.south) -- ($ (xy.north) + (-0.1,0) $) ;
}
\ifthen{\equal{#3}{0}}
{\node[inner sep=1.75pt] (wxy) at ($ (wx.north) + (0.5,0) $) {$\scriptstyle\times$} ;}
%\ifthen{\equal{#1}{1}}
% {\ifthen{\equal{#2}{1}}
% {
\ifthen{\equal{#4}{1}}
{
\node (uvwx) at ($ (uvw.north) + (0.5,0.3) $) {} ;
\draw ($ (uvw.north) + (0.1,0) $) -- (uvwx.south) -- ($ (vwx.north) + (-0.1,0) $) ;
}
\ifthen{\equal{#4}{0}}
{\node[inner sep=1.75pt] (uvwx) at ($ (uvw.north) + (0.5,0) $) {$\scriptstyle\times$} ;}
% }
% }
%\ifthen{\equal{#2}{1}}
% {\ifthen{\equal{#3}{1}}
% {
\ifthen{\equal{#5}{1}}
{
\node (vwxy) at ($ (vwx.north) + (0.5,0.3) $) {} ;
\draw ($ (vwx.north) + (0.1,0) $) -- (vwxy.south) -- ($ (wxy.north) + (-0.1,0) $) ;
}
\ifthen{\equal{#5}{0}}
{\node[inner sep=1.75pt] (vwxy) at ($ (vwx.north) + (0.5,0) $) {$\scriptstyle\times$} ;}
% }
% }
%\ifthen{\equal{#1}{1}}
% {\ifthen{\equal{#2}{1}}
% {\ifthen{\equal{#3}{1}}
% {\ifthen{\equal{#4}{1}}
% {\ifthen{\equal{#5}{1}}
% {
\ifthen{\equal{#6}{1}}
{
\node (uvwxy) at ($ (uvwx.north) + (0.5,0.3) $) {} ;
\draw ($ (uvwx.north) + (0.1,0) $) -- (uvwxy.south) -- ($ (vwxy.north) + (-0.1,0) $) ;
}
\ifthen{\equal{#6}{0}}
{\node[inner sep=1.75pt] (uvwxy) at ($ (uvwx.north) + (0.5,0) $) {$\scriptstyle\times$} ;}
% }
% }
% }
% }
% }
\end{tikzpicture}\!
}

\title{Coherent presentations of monoids with a right-noetherian Garside
family}

\author{Pierre-Louis Curien}

\address[Pierre-Louis Curien]{Université Paris Cité, CNRS, Inria, IRIF, F-75013, Paris, France}

\email{curien@irif.fr}

\author{Alen \DJ uri\'{c}}

\address[Alen \DJ uri\'{c}]{Université Paris Cité, CNRS, Inria, IRIF, F-75013, Paris, France}

\email{alen.djuric@protonmail.com}

\thanks{%
\noindent\begin{minipage}[b][1\totalheight][t]{0.075\columnwidth}%
\begin{tikzpicture}[scale=0.017, every node/.style={scale=0.017}] \fill[fill={rgb,255:red,0;green,51;blue,153}] (-27,-18) rectangle (27,18); \pgfmathsetmacro\inr{tan(36)/cos(18)} \foreach \i in {0,1,...,11} { \begin{scope}[shift={(30*\i:12)}] \fill[fill={rgb,255:red,255;green,204;blue,0}] (90:2) \foreach \x in {0,1,...,4} { -- (90+72*\x:2) -- (126+72*\x:\inr) }; \end{scope} } \end{tikzpicture}%
\end{minipage}%\hfill{}%
\begin{minipage}[b][1\totalheight][t]{0.8\columnwidth}%
The second author has received funding from the European Union's Horizon
2020 research and innovation programme under the Marie Sk\l odowska-Curie
grant agreement No. 754362.%
\end{minipage}}

\author{Yves Guiraud}

\address[Yves Guiraud]{Université Paris Cité and Sorbonne Université, CNRS, INRIA, IMJ-PRG,
F-75013, Paris, France }

\email{yves.guiraud@imj-prg.fr}

\date{30 November 2022}
\begin{abstract}
This paper shows how to construct coherent presentations (presentations
by generators, relations and relations among relations) of monoids
admitting a right-noetherian Garside family. Thereby, it resolves
the question of finding a unifying generalisation of the following
two distinct extensions of construction of coherent presentations
for Artin-Tits monoids of spherical type: to general Artin-Tits monoids,
and to Garside monoids. The result is applied to some monoids which
are neither Artin-Tits nor Garside. 

\end{abstract}



\maketitle
\tableofcontents{}

\section{Introduction}

\subsection{Coherent presentations of monoids}

A monoid can be presented by a generating set and a set of relations
between words over the generating set. A coherent presentation of
a monoid consists of a set of generators, a set of generating relations,
and a set of generating relations among relations, having the property
that, for every pair of parallel sequences of relations, there is
a relation among relations between those two sequences.

Coherent presentations generalise $2$-syzygies for presentations
of groups. They form the first dimensions of polygraphic resolutions
of monoids, from which abelian resolutions can be deduced. For motivation
and context of the notion of coherent presentations, we refer the
reader to~\cite{GGM}.  In particular, it has been proved in~\cite{GGM}
that Deligne's characterisation~\cite{Del} of the weak actions of
an Artin-Tits monoid $B^{+}\left(W\right)$ of spherical type on categories
is equivalent to constructing a certain coherent presentation, denoted
$\Gar_{3}\left(W\right)$ in~\cite{GGM}, of $B^{+}\left(W\right)$.
This construction has been extended in~\cite{GGM}, using methods
from rewriting theory, in two disjoint directions: to general Artin
Tits monoids, and to Garside monoids. Coherent presentations are also
studied in~\cite{EW}, under the name $3$-presentations.

\subsection{Rewriting methods}

Generating relations, when considered directed from left to right
(i.e. as ordered pairs), provide rewriting rules. A presentation
is called terminating if there is no infinite rewriting sequence;
it is called confluent if any two distinct rewriting sequences starting
from the same word can be completed in such a way that they eventually
lead to a common result; it is convergent if it is both terminating
and confluent. A homotopical completion-reduction procedure, developed
in~\cite{GGM}, enriches a terminating presentation to a coherent
one. The main element is Squier's theorem, which allows one to simply
compute generators of the relations among relations for a convergent
presentation. This\uline{} procedure has three stages. Firstly,
a Knuth-Bendix completion procedure enriches a terminating presentation
to a convergent one by adding a (not necessarily finite) number of
relations. Secondly, a Squier completion procedure adjoins relations
among relations, thus providing a coherent presentation of the monoid
admitting the starting presentation. Thirdly, a homotopical reduction
procedure removes redundant relations. These homotopical transformations
of presentations having certain properties are illustrated by the
following diagram and recollected in Section~\ref{sec:homotopical-transformations}.
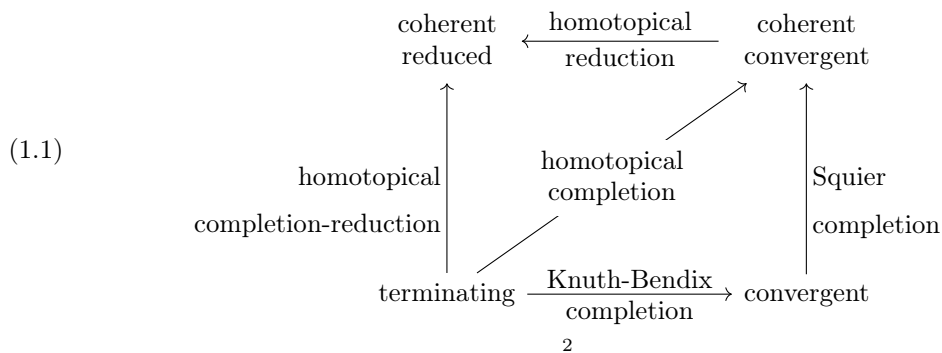
\begin{equation}\label{eq:transformations}
\begin{tikzcd}[row sep=huge, column sep=large] \begin{array}{c}\textrm{coherent}\\ \textrm{reduced} \end{array} & & \begin{array}{c}\textrm{coherent}\\ \textrm{convergent} \end{array} \arrow[ll, "\textrm{\normalsize homotopical}"', "\textrm{\normalsize reduction}"] \\ \\ \textrm{terminating} \arrow[rruu, "\begin{array}{c}\textrm{\normalsize homotopical}\\ \textrm{\normalsize completion} \end{array}" description] \arrow[rr, "\textrm{\normalsize Knuth-Bendix}", "\textrm{\normalsize completion}"'] \arrow[uu, "\textrm{\normalsize homotopical}", "\textrm{\normalsize completion-reduction}" near start]  & & \textrm{convergent} \arrow[uu, "\textrm{\normalsize Squier}"', "\textrm{\normalsize completion}"' near start] \end{tikzcd}
\end{equation}

Let us also illustrate the second stage by giving a preview of Example
\ref{exa:klein-bottle-completion}. Consider the following convergent
presentation of the Klein bottle monoid:
\[
\left\langle a,b\,\middle\vert\,bab\stackrel{\alpha}{\rightarrow}a,baa\stackrel{\beta}{\rightarrow}aab\right\rangle .
\]
There are exactly two critical branchings, i.e. minimal overlaps
of the rewriting steps: $\left\{ \alpha ab,ba\alpha\right\} $ and
$\left\{ \alpha aa,ba\beta\right\} $. Both branchings are confluent.
A Squier completion procedure adds the generators $A$ and $B$ of
the relations among relations. Here are the shapes of $A$ and $B$:
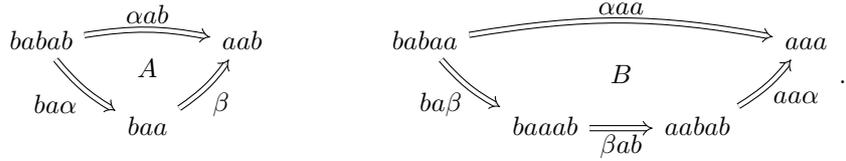
\[
\begin{tikzcd}[row sep=small, column sep=small] 
babab \arrow[Rightarrow, rr, bend left=10, "\alpha ab"{name=U}] \arrow[Rightarrow, rdd, bend right=10, "ba\alpha"'] & & aab \\ \\ & baa \arrow[from=U, phantom, "A"] \arrow[Rightarrow, ruu, bend right=10, "\beta"'] \end{tikzcd}
%\hskip \textwidth minus \textwidth 
\qquad\qquad
\begin{tikzcd}[row sep=small, column sep=small] 
babaa \arrow[Rightarrow, rrrr, bend left=10, "\alpha aa"{name=U}] \arrow[Rightarrow, rdd, bend right=10, "ba\beta"'] & & & & aaa \\ \\ & baaab \arrow[Rightarrow, rr, "\beta ab"'{name=D}] & & aabab \arrow[from=U, to=D, phantom, "B"] \arrow[Rightarrow, ruu, bend right=10, "aa\alpha"'] \end{tikzcd}.
\]

In~\cite{GGM}, Gaussent, the third author and Malbos have performed
a homotopical completion-reduction procedure to compute coherent presentations
of two disjoint generalisations of Artin-Tits monoids of spherical
type: general Artin-Tits monoids, and Garside monoids. We recall those
two generalisations in Subsection~\ref{subsec:ggm-coherent-presentations}
as Examples~\ref{exa:artin-tits-coherent} and~\ref{exa:garside-monoid-coherent},
respectively. In~\cite{GMM}, the third author, Malbos and Mimram
have computed coherent presentations of plactic and Chinese monoids
by applying a homotopical completion-reduction procedure.

\subsection{Garside families}

A Garside family in a monoid is a generating family, not minimal in
general, but ensuring some desirable properties. Namely, the notion
of a Garside family~\cite{DDM} is a result of successive generalisations
to wider classes of monoids of a particular type of normal form, first
implicitly hinted in braid monoids by Garside~\cite{Gar} in 1969,
known as the greedy normal form. In particular, it generalises Artin-Tits
monoids and Garside monoids. The greedy normal form is easily computed
as it has very nice locality properties. These notions are recalled
in Section~\ref{sec:garside-families}.

Garside~\cite{Gar} investigated arithmetic properties of braid groups.
He solved the word problem and the conjugacy problem in braid groups
by introducing braid monoids. Among other things, he proved that
the braid monoid $B_{n}^{+}$ is left-cancellati\-ve, and that any
two elements of $B_{n}^{+}$ admit a least common multiple. He also
introduced the Garside element (he called it the fundamental word)
of a braid monoid.

Garside's observations for braid monoids were generalised to Artin-Tits
monoids of spherical type by Brieskorn and Saito~\cite{BS}, and by
Deligne who later explicitly gave Garside's presentation for Artin-Tits
monoids of spherical type in~\cite{Del}. Michel~\cite{Mic} extended
this presentation to all Artin-Tits monoids.

The greedy normal form was later generalised to Artin-Tits monoids,
based on Garside's observations (see~\cite[Introduction]{DDGKM} for
references). Dehornoy and Paris~\cite{DP} introduced Garside monoids
in order to abstract properties which establish the existence of the
greedy normal form. Dehornoy, Digne and Michel~\cite{DDM} further
generalised Garside monoids to categories admitting Garside families
(as recalled for monoids in Subsection~\ref{subsec:garside-family}
here). A thorough development of the notion of a Garside family can
be found in the book~\cite{DDGKM}. Dehornoy and the third author
\cite{DG} introduced monoids admitting quadratic normalisations,
thereby generalising monoids admitting Garside families. We refer
the reader to the survey~\cite{Deh} for an overview of the successive
extensions of the greedy normal form from braid monoids to monoids
admitting left-weighted quadratic normalisations.

\subsection{Contributions}

The objective of the present paper is to unify the two above-mentioned
results of~\cite{GGM} in the same generalisation. Namely, we apply
a homotopical completion-reduction procedure to compute coherent presentations
of a certain class of monoids admitting a Garside family. Our present
contribution has the following two main steps.
\begin{enumerate}
\item First, we use the fact that every left-cancellative monoid $M$ containing
no nontrivial invertible element and for every Garside family $S$
in $M$, there is a presentation, here denoted $\Gar_{2}\left(S\right)$,
having $S\setminus\left\{ 1\right\} $ as generating set, with generating
relations $\alpha$ of the form $s|t=st$, for $s,t\in S\setminus\left\{ 1\right\} $
with $st\in S$ (Proposition~\ref{prop:dg-6.3.1}, adapted from~\cite{DG}).
We observe (Example~\ref{exa:garside's-presentation-artin-tits-monoid})
that Garside's presentation $\Gar_{2}\left(W\right)$ of an Artin-Tits
monoid $B^{+}\left(W\right)$ is a special case, with $S$ being the
Coxeter group $W$. Similarly (Example~\ref{exa:garside's-presentation-garside-monoid}),
Garside's presentation $\Gar_{2}\left(M\right)$ of a Garside monoid
$M$ is another special case of $\Gar_{2}\left(S\right)$.
\item Then, starting from $\Gar_{2}\left(S\right)$, we embark on extending
\cite[Theorem~3.1.3]{GGM} (which we recall in Example~\ref{exa:artin-tits-coherent})
to a wider class of monoids, including left-cancellative noetherian
monoids containing no nontrivial invertible element, admitting a Garside
family. Working in a more general setting, we encounter additional
critical branchings which cannot occur in the case of Artin-Tits or
Garside monoids due to specific properties not shared by Garside families
in general. Therefore, we construct new generating relations among
relations. Conveniently, we then remove all the additional relations
using the homotopical reduction procedure. 
\end{enumerate}
This results in Theorem~\ref{thm:coherent-presentation-garside-family},
our main result,  of which we give here a weaker, but simpler version
(our Corollary~\ref{cor:coherent-presentation-noetherian}). 
\begin{thm*}
Assume that $M$ is a left-cancellative noetherian monoid containing
no nontrivial invertible element, and $S\subseteq M$ is a Garside
family containing $1$. Then $M$ admits the coherent presentation
$\Gar_{3}\left(S\right)$ which extends $\Gar_{2}\left(S\right)$
with the following set of generating relations among relations:\[
\begin{tikzcd}[%cramped, 
row sep=scriptsize, column sep=scriptsize, matrix scale=1, transform shape, nodes={scale=1}] & uv|w \arrow[Rightarrow, rd, bend left=10, "\alpha_{uv,w}"] \arrow[dd, phantom, "A_{u,v,w}"] & \\ u|v|w \arrow[Rightarrow, ru, bend left=10, "\alpha_{u,v}| w"] \arrow[Rightarrow, rd, bend right=10, "u|\alpha_{v,w}"'] & & uvw \\ & u|vw \arrow[Rightarrow, ru, bend right=10, "\alpha_{u,vw}"'] \end{tikzcd},
\]for all $u,v,w\in S\setminus\left\{ 1\right\} $ such that $uv,vw,uvw\in S$.
\end{thm*}
Note that $A_{u,v,w}$ can be read as a relation ensuring associativity.
We shall reach $\Gar_{3}\left(S\right)$ by applying a homotopical
completion-reduction procedure to the presentation $\Gar_{2}\left(S\right)$.

In Section~\ref{sec:examples}, the result is used to compute coherent
presentations of some monoids which are neither Artin-Tits nor Garside,
and to construct a finite coherent presentation of the Artin-Tits
monoid of type $\widetilde{A}_{2}$, taking a finite generating set.
In some cases, homotopical reduction can be carried further: as a
matter of fact, in Subsection~\ref{subsec:non-spherical}, we prove
that Artin's presentation of the Artin-Tits monoid of type $\widetilde{A}_{2}$
is coherent (with the empty set of generating relations among relations).

We mainly consider monoids because that is where our applications
lie, but the approach presented here can be extended to categories.

\subsection{Acknowledgements}

The authors would like to thank the anonymous reviewer(s) for his/her/their helpful comments; they greatly helped us to improve the quality of this article.

\section{\label{sec:presentations-categories-polygraphs}Presentations of
monoids by polygraphs}

In this section, we briefly recall the notions concerning polygraphic
presentations of monoids (technical elaboration whereof can be found
in~\cite{GGM}). Basic terminology is given in Subsection~\ref{subsec:presentations-by-polygraphs}.
Some basic notions of polygraphic rewriting theory are recollected
in Subsection~\ref{subsec:rewriting-2-polygraphs}. Subsection~\ref{subsec:coherent-presentations}
recalls the notion of coherent presentation.

Throughout the present article, $2$-categories and $3$-categories
are always assumed to be strict (see e.g.~\cite[Section~2]{GM2}).
In diagrams, distinct arrows are used to denote $k$-cells for low
$k$: $\rightarrow$, $\Rightarrow$, $\Rrightarrow$ for $k$ equal
to $1$, $2$ and $3$, respectively.

\subsection{\label{subsec:presentations-by-polygraphs}Presentations by $2$-polygraphs}

Polygraphs encompass words, rewriting rules, and homotopical properties
of the rewriting systems in the same globular object. They provide
a generalisation of a presentation of a monoid by generators and relations
to the higher categories which are free up to codimension $1$.

A polygraph is a higher-dimensional generalisation of a graph. Recall
that a (directed) graph is a pair $\left(X_{0},X_{1}\right)$ of sets,
together with two maps, called source and target, from $X_{1}$ to
$X_{0}$. A $0$-polygraph $\left(X_{0}\right)$ is a set, a \textbf{$1$-polygraph}
$\left(X_{0},X_{1}\right)$ is a graph. The free category generated
by a $1$-polygraph $\left(X_{0},X_{1}\right)$ is denoted by $X_{1}^{\ast}$.
A \textbf{$2$-polygraph} is a triple $X=\left(X_{0},X_{1},X_{2}\right)$,
where $\left(X_{0},X_{1}\right)$ is a $1$-polygraph and $X_{2}$
is a set of \textbf{$1$-spheres}, i.e. pairs of parallel paths, in
$X_{1}^{\ast}$.

For a $2$-polygraph $X$, the \textbf{category presented by $X$},
denoted $\overline{X}=X_{1}^{\ast}/X_{2}$, is obtained by factoring
out generating $2$-cells, regarded as relations among $1$-cells
of $X_{1}^{\ast}$. For a monoid $M$, a \textbf{presentation of }$M$
is a $2$-polygraph $X$ such that $M$ is isomorphic to $\overline{X}$.
In this case, which we are mainly interested in, $X_{0}$ is a singleton
so any pair of paths in $X_{1}^{\ast}$ forms a $1$-sphere, and~$X_{1}^{\ast}$
is the free monoid generated by the set $X_{1}$. Elements of $X_{k}$
are called \textbf{generating $k$-cells}.
\begin{example}[The standard presentation]
\label{exa:standard-presentation}Let \textbf{$M$} be a monoid.
The \textbf{standard presentation} of $M$ is the $2$-polygraph $\Std_{2}\left(M\right)$
consisting of:
\begin{itemize}
\item one generating $0$-cell $x$;
\item a generating $1$-cell $\widehat{u}$ for every element $u$ of $M$;
\item a generating $2$-cell $\gamma_{u,v}:\widehat{u}\widehat{v}\Rightarrow\widehat{uv}$
for every pair of elements $u$ and $v$ of $M$;
\item one generating $2$-cell $\iota_{x}:1_{x}\Rightarrow\widehat{1_{x}}$
.
\end{itemize}
\end{example}

\subsection{\label{subsec:rewriting-2-polygraphs}Rewriting properties of $2$-polygraphs}

Let us adopt some basic terminology from string rewriting. If $S$
is a set, $S^{\ast}$ denotes the free monoid over $S$. Elements
of $S$ and $S^{\ast}$ are respectively called letters and \textbf{words}.
We write $u|v$ for the concatenation of two words $u$ and $v$,
sometimes omitting the separation symbol when that does not cause
ambiguity. Let $M$ be a monoid generated by a set $S$. A \textbf{normal
form} for $M$ with respect to $S$ is a set-theoretic section of
the evaluation map (canonical projection) $\ev:S^{\ast}\to M$. In
other words, a normal form maps elements of $M$ to distinguished
representative words. A word $s_{1}|\cdots|s_{p}$ is said to be a
\textbf{decomposition} of an element $f$ of $M$ if the equality
$s_{1}\cdots s_{p}=f$ holds in $M$.

Assume that a $2$-polygraph $X$ is a presentation of a monoid $M$.
Generating $2$-cells of $X$ are called \textbf{rewriting rules}.
The \textbf{free $2$-category over $X$}, denoted $X_{2}^{\ast}=X_{1}^{\ast}\left[X_{2}\right]$,
is obtained by adjoining to $X_{1}^{\ast}$ all the formal compositions
of elements of $X_{2}$, treated as formal $2$-cells. Standard notions
from rewriting theory naturally translate into the framework of polygraphs.
A \textbf{rewriting step} of a $2$-polygraph $X$ is a $2$-cell
of the free category $X_{2}^{\ast}$ which contains a single generating
$2$-cell of $X$, here considered as a transformation of its source
into its target. So, a rewriting step has a shape
\[
\begin{tikzcd} \bullet \arrow[r, "w"] & \bullet \arrow[r, shorten <=-2pt,  shorten >=-2pt, bend left=50, "u"] \arrow[r, phantom, bend left=50, ""{name=U, below}] \arrow[r, shorten <=-2pt,  shorten >=-2pt, bend right=50, "v"'] \arrow[r, phantom, bend right=50, ""{name=D}] & \bullet \arrow[Rightarrow, from=U, to=D, shorten <=2pt, shorten >=2pt, "\alpha"] \arrow[r, shorten <=-2pt,  shorten >=-2pt, "w'"] & \bullet \end{tikzcd},
\]
where $\alpha:u\Rightarrow v$ is a generating $2$-cell of $X$,
and $w$ and $w'$ are $1$-cells of $X_{2}^{\ast}$, and the $0$-cell
is denoted by $\bullet$.

Let $u$ and $v$ be $1$-cells of $X_{2}^{\ast}$. It is said that
$u$ \textbf{rewrites} to $v$ if there is a finite composable sequence
of rewriting steps with source $u$ and target $v$. A $1$-cell $u$
is \textbf{reduced} if there is no rewriting step whose source is
$u$. 

Let $X$ be a $2$-polygraph. A \textbf{termination order} on $X$
is a well-founded order relation $\leq$ on parallel $1$-cells of
$X_{2}^{\ast}$ enjoying the following properties:
\begin{itemize}
\item the compositions by $1$-cells of $X_{2}^{\ast}$ are strictly monotone
in both arguments, i.e. $\leq$ is compatible with the composition
of $1$-cells;
\item for every generating $2$-cell $\alpha$ of $X$, the strict inequality
$s\left(\alpha\right)>t\left(\alpha\right)$ holds.
\end{itemize}
A $2$-polygraph $X$ is \textbf{terminating} if it has no infinite
sequence of rewriting steps. Admitting a termination order is equivalent
to being terminating (in a terminating polygraph, a termination order
is obtained by putting $u>v$ for $1$-cells $u$ and $v$ if $u$
rewrites to $v$).

A \textbf{branching} of a $2$-polygraph $X$ is an unordered pair
$\left\{ \alpha,\beta\right\} $ of sequences of rewriting steps of
$X_{2}^{\ast}$ having the same source, called the source of branching.
If $\alpha$ and $\beta$ are rewriting steps, a branching $\left\{ \alpha,\beta\right\} $
is called \textbf{local}. A local branching is \textbf{trivial} if
it has one of the following two shapes: $\left\{ \alpha,\alpha\right\} $,
or $\left\{ \alpha v,u\beta\right\} $ for $u=s\left(\alpha\right)$
and $v=s\left(\beta\right)$. Local branchings can be compared by
the order $\preccurlyeq$ generated by the relations $\left\{ \alpha,\beta\right\} \preccurlyeq\left\{ u\alpha v,u\beta v\right\} $
given for every local branching $\left\{ \alpha,\beta\right\} $ and
all possible $1$-cells $u$ and $v$ of $X_{2}^{\ast}$. A minimal
nontrivial local branching is called \textbf{critical}. A branching
$\left\{ \alpha,\beta\right\} $ is confluent if $\alpha$ and $\beta$
can be completed into sequences having the same target. A $2$-polygraph
$X$ is \textbf{confluent} (resp.\ locally confluent, resp.\ critically
confluent) if all its branchings (resp.\ local branchings, resp.\ critical
branchings) are confluent. If $X$ is terminating and confluent, it
is called \textbf{convergent}. A convergent $2$-polygraph $X$ is
called a \textbf{convergent presentation} of any category isomorphic
to $\overline{X}$. In that case, for every $1$-cell $u$ of~$X^{\ast}$,
there is a unique reduced word, denoted by $\widehat{u}$, to which
$u$ rewrites.

Two basic results of rewriting theory concerning confluence, called
Newman's lemma~\cite[Theorem~3]{New} and the critical branchings
theorem respectively, are also valid for polygraphs.
\begin{thm}[{\cite[Theorem~3.1.6]{GM2}}]
\label{thm:newman-critical}Let $X$ be a $2$-polygraph.
\begin{enumerate}
\item If $X$ is terminating, then $X$ is confluent if, and only if, it
is locally confluent.
\item $X$ is locally confluent if, and only if, it is critically confluent.
\end{enumerate}
\end{thm}

As a consequence of Theorem~\ref{thm:newman-critical}, a $2$-polygraph
is convergent if, and only if, it is terminating and its critical
branchings are confluent.
\begin{example}
Consider the free abelian monoid:
\begin{equation}
\mathbb{N}^{3}=\left\langle a,b,c\,\middle\vert\,ba\stackrel{\alpha}{\Rightarrow}ab,cb\stackrel{\beta}{\Rightarrow}bc,ca\stackrel{\gamma}{\Rightarrow}ac\right\rangle .\label{eq:free_abelian_N^3}
\end{equation}

This presentation~\eqref{eq:free_abelian_N^3} admits the following
termination order: comparing the lengths of words, then applying lexicographic
order, induced by $a<b<c$, if words have the same length. Hence,
it is terminating.

Let us illustrate confluence of~\eqref{eq:free_abelian_N^3} on the
unique critical branching $\left\{ \beta a,c\alpha\right\} $:
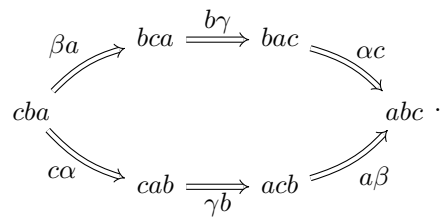
\[
\begin{tikzcd}[row sep=scriptsize, column sep=normal] & bca \arrow[Rightarrow, r, "b\gamma"] & bac \arrow[Rightarrow, rd, bend left=15, "\alpha c"] \\ cba \arrow[Rightarrow, ru, bend left=15, "\beta a"] \arrow[Rightarrow, rd, bend right=15, "c\alpha"'] & & & abc \\ & cab \arrow[Rightarrow, r, "\gamma b"'] & acb \arrow[Rightarrow, ru, bend right=15, "a\beta"'] \end{tikzcd}.
\]

Thus, the presentation~\eqref{eq:free_abelian_N^3} is convergent,
by Theorem~\ref{thm:newman-critical}.
\end{example}

\subsection{\label{subsec:coherent-presentations}Coherent presentations}

A $2$-category (resp.\ $3$-category) is called a $\left(2,1\right)$-category
(resp.\ a $\left(3,1\right)$-category) if its $2$-cells (resp.\ $2$-cells
and $3$-cells) are invertible. For a $2$-polygraph $X$, the \textbf{free
$\left(2,1\right)$-category over $X$}, denoted $X_{2}^{\top}=X_{1}^{\ast}\left(X_{2}\right)$,
is constructed by adjoining to $X_{1}^{\ast}$ all the formal compositions
of elements of $X_{2}$ and formal inverses of elements of $X_{2}$,
and then factoring out the compositions of elements with their corresponding
inverses. A \textbf{$\left(3,1\right)$-polygraph} is a quadruple
$X=\left(X_{0},X_{1},X_{2},X_{3}\right)$, where $\left(X_{0},X_{1},X_{2}\right)$
is a $2$-polygraph and $X_{3}$ is a set of \textbf{$2$-spheres},
i.e. pairs of parallel paths of $2$-cells, in $X_{2}^{\top}$. For
a $\left(3,1\right)$-polygraph $X$, the \textbf{free $\left(3,1\right)$-category
over $X$}, denoted $X_{3}^{\top}=X_{2}^{\top}\left(X_{3}\right)$,
is constructed by adjoining to $X_{2}^{\top}$ all the formal compositions
of elements of $X_{3}$ and formal inverses of elements of $X_{3}$,
and then factoring out the compositions of elements with their corresponding
inverses. A $\left(3,1\right)$-polygraph is called convergent if
its underlying $2$-polygraph is. The \textbf{category presented by
a $\left(3,1\right)$-polygraph} $X$ is again $\overline{X}$, the
category presented by its underlying $2$-polygraph. An \textbf{extended
presentation of a monoid }$M$ is a $\left(3,1\right)$-polygraph
$X$ such that $M$ is isomorphic to $\overline{X}$. 

\begin{defn}
A \textbf{coherent presentation of} a monoid \textbf{$M$} is an extended
presentation $\left(X_{0},X_{1},X_{2},X_{3}\right)$ of $M$ such
that factoring out elements of $X_{3}$, leaves only trivial $2$-spheres
(where the parallel paths are equal).
\end{defn}

\begin{example}[The standard coherent presentation]
Let us extend $\Std_{2}\left(M\right)$ from Example~\ref{exa:standard-presentation}
with the following $3$-cells
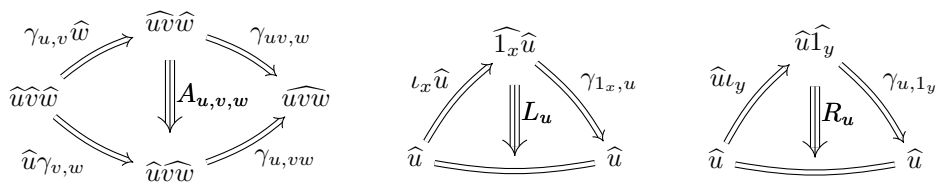
\[
\begin{tikzcd}[row sep=scriptsize] & \widehat{uv}\widehat{w} \arrow[Rightarrow, rd, bend left=10, "\gamma_{uv,w}\\"] \tarrow[shorten <=7pt, shorten >=7pt, "A_{u,v,w}"]{dd} & \\ \widehat{u}\widehat{v}\widehat{w} \arrow[Rightarrow, ru, bend left=10, "\gamma_{u,v}\widehat{w}"] \arrow[Rightarrow, rd, bend right=10, "\widehat{u}\gamma_{v,w}"'] & & \widehat{uvw} \\ & \widehat{u}\widehat{vw} \arrow[Rightarrow, ru, bend right=10, "\gamma_{u,vw}"'] \end{tikzcd}
%\hskip \textwidth minus \textwidth 
\qquad
\begin{tikzcd}[row sep=scriptsize, column sep=scriptsize] & \widehat{1_{x}}\widehat{u} \arrow[Rightarrow, rdd, bend left=10,  shorten <=-2pt, shorten >=-2pt, "\gamma_{1_{x},u}"{name=U}] & \\ \\ \widehat{u} \arrow[Rightarrow, ruu, bend left=10,  shorten <=-2pt, shorten >=-2pt, "\iota_{x}\widehat{u}"] \arrow[equal, rr, bend right=10, ""{name=D}] & & \widehat{u} \tarrow[from=1-2, to=D, shorten <=7pt, shorten >=1pt, "L_{u}"]{} \end{tikzcd}
%\hskip \textwidth minus \textwidth 
\qquad
\begin{tikzcd}[row sep=scriptsize, column sep=scriptsize] & \widehat{u}\widehat{1_{y}} \arrow[Rightarrow, rdd, bend left=10,  shorten <=-2pt, shorten >=-2pt, "\gamma_{u,1_{y}}"{name=U}] & \\ \\ \widehat{u} \arrow[Rightarrow, ruu, bend left=10,  shorten <=-2pt, shorten >=-2pt, "\widehat{u}\iota_{y}"] \arrow[equal, rr, bend right=10, ""{name=D}] & & \widehat{u} \tarrow[from=1-2, to=D, shorten <=7pt, shorten >=1pt, "R_{u}"]{} \end{tikzcd}
\]
for every triple $u$, $v$, $w$ of elements of $M$. The resulting
$\left(3,1\right)$-polygraph, denoted by $\Std_{3}\left(M\right)$,
is called the \textbf{standard coherent presentation} of $M$ (see
\cite[Subsection~3.3.3]{Gui} for the explanation why $\Std_{3}\left(M\right)$
is, indeed, a coherent presentation).
\end{example}

\section{\label{sec:homotopical-transformations}Homotopical transformations
of polygraphs}

This section elaborates the diagram~\eqref{eq:transformations}, by
recalling the notion of homotopical completion-reduction, introduced
in~\cite{GGM}. Subsection~\ref{subsec:knuth-bendix} recollects the
Knuth-Bendix completion procedure which transforms a terminating $2$-polygraph
into a convergent one. Subsection~\ref{subsec:squier-completion}
recalls the Squier completion procedure which upgrades a convergent
$2$-polygraph to a convergent coherent $\left(3,1\right)$-polygraph.
In Subsection~\ref{subsec:reduction}, we report on the homotopical
reduction procedure which turns a coherent $\left(3,1\right)$-polygraph
into a coherent one having fewer generating cells. Finally, Subsection
\ref{subsec:special-case} describes a particular method for obtaining
a homotopical reduction in case when the starting coherent $\left(3,1\right)$-polygraph
is also convergent.

\subsection{\label{subsec:knuth-bendix}Knuth-Bendix completion}

Starting with a terminating $2$-polygraph $X$, equi\-pped with
a total termination order $\leq$, the Knuth-Bendix completion procedure
adjoins generating $2$-cells aiming to produce a convergent $2$-polygraph,
which presents a category presented by $X$. It works by iteratively
examining all the critical branchings and adjoining a new generating
$2$-cell whenever the branching is not already confluent. Namely,
for a critical branching $\left\{ \alpha,\beta\right\} $, if $\widehat{t\left(\alpha\right)}>\widehat{t\left(\beta\right)}$
(resp.\ $\widehat{t\left(\beta\right)}>\widehat{t\left(\alpha\right)}$),
a generating $2$-cell $\gamma:\widehat{t\left(\alpha\right)}\Rightarrow\widehat{t\left(\beta\right)}$
(resp.\ $\gamma:\widehat{t\left(\beta\right)}\Rightarrow\widehat{t\left(\alpha\right)}$)
is adjoined, thus forcing the confluence of the branching: 
\[
\begin{tikzcd}[row sep=scriptsize, column sep=normal] & t\left(\alpha\right) \arrow[Rightarrow, r] & \widehat{t\left(\alpha\right)} \arrow[Rightarrow, dd, "\gamma", dashed] \\ \ast \arrow[Rightarrow, ru, bend left=15, "\alpha"] \arrow[Rightarrow, rd, bend right=15, "\beta"'] & & \\ & t\left(\beta\right) \arrow[Rightarrow, r] & \widehat{t\left(\beta\right)} \end{tikzcd}.
\]
If new critical branchings are created by adjoining additional generating
$2$-cells, confluence of such critical branchings is examined. For
details, see~\cite[3.2.1]{GM2}. This procedure is not guaranteed
to terminate. In fact, its termination depends on the chosen termination
order (see~\cite[Example~6.3.1]{ECHLPT}). If it does terminate, the
result is a convergent $2$-polygraph. Otherwise, it produces an increasing
sequence of $2$-polygraphs, and the result is the union of this sequence.
Either way, the result is called a \textbf{Knuth-Bendix completion}
of $X$. Note that different orders of examining critical branchings
may result in different $2$-polygraphs.
\begin{thm}[{\cite[Theorem~3.2.2]{GM2}}]
\label{thm:knuth-bendix}Assume that $X$ is a $2$-polygraph, equipped
with a total termination order, presenting a monoid $M$. Then every
Knuth-Bendix completion of $X$ is a convergent presentation of $M$.
\end{thm}

\begin{rem}
\label{rem:alternative-knuth-bendix}The Knuth-Bendix completion
procedure, as described above, requires not only termination, but
also the presence of a total termination order, to be able to orient
the generating $2$-cells which are added, and to be able to maintain
the termination during the completion. There is an alternative approach.
Namely, we can orient the newly added generating $2$-cells \textquotedbl{}by
hand\textquotedbl{}, according to our inspiration, and verify after
each addition in an ad hoc manner whether we maintain a terminating
presentation, without having defined a total order at the beginning
(we shall do this in the proof of Proposition~\ref{prop:completion-assuming-termination}).
Therefore, we can invoke Theorem~\ref{thm:knuth-bendix} even if we
do not provide a total order, as long as we are able to ensure termination
after each addition of a generating $2$-cell (we shall do this in
the proof of Corollary~\ref{cor:convergent-presentation-garside-family}).
\end{rem}

\subsection{\label{subsec:squier-completion}Squier completion}

A \textbf{family of generating confluences} of a convergent $2$-polygraph~$X$ is a set of $2$-spheres, treated as formal $3$-cells, in $X_{2}^{\top}$
containing, for every critical branching $\left\{ \alpha,\beta\right\} $
of $X$, exactly one $3$-cell $A$:\[
\begin{tikzcd}[row sep=scriptsize, column sep=large] & \ast \arrow[Rightarrow, rd, bend left=15, shorten <=-2pt, shorten >=-2pt, "\alpha'"] \tarrow[shorten <=5pt, shorten >=3pt, dashed, "A"]{dd} & \\ \ast \arrow[Rightarrow, ru, bend left=15, shorten <=-2pt, shorten >=-2pt, "\alpha"] \arrow[Rightarrow, rd, bend right=15, shorten <=-2pt, shorten >=-2pt, "\beta"'] & & \ast \\ & \ast \arrow[Rightarrow, ru, bend right=15, shorten <=-2pt, shorten >=-2pt, "\beta'"'] \end{tikzcd},
\] where $\alpha'$ and $\beta'$ are completing $\alpha$ and $\beta$,
respectively, into sequences having the same target (such $\alpha'$
and $\beta'$ exist by the assumption of confluence).

A \textbf{Squier completion} of a convergent $2$-polygraph $X$ is
a $\left(3,1\right)$-polygraph with X as underlying $2$-polygraph,\uline{}
whose generating $3$-cells form a family of generating confluences
of $X$. The following result is due to Squier; we state a version
in terms of polygraphs and higher-dimensional categories proved in
\cite{GM2}.
\begin{thm}[{\cite[Theorem~4.3.2]{GM2}}]
\label{thm:squier}If $X$ is a convergent presentation of a monoid
$M$, then every Squier completion of $X$ is a convergent coherent
presentation of $M$.
\end{thm}

Theorem~\ref{thm:squier} is extended to higher-dimensional polygraphs
in~\cite[Proposition~4.3.4]{GM1}.

Let $X$ be a terminating $2$-polygraph equipped with a total termination
order $\leq$. A homotopical completion of $X$ is a Squier completion
of a Knuth-Bendix completion of $X$. We have seen that a Knuth-Bendix
completion procedure enriches a terminating $2$-polygraph to a convergent
one, and that the Squier completion of a convergent $2$-polygraph
$X$ is a coherent presentation of $\overline{X}$. Those two transformations
can be performed consecutively. They can also be performed simultaneously
(see~\cite[2.2.4]{GGM}). The result is called a \textbf{homotopical
completion} of $X$. Theorem~\ref{thm:squier} has the following
consequence.
\begin{thm}
\label{thm:homotopical-completion}Assume that a $2$-polygraph $X$
is a terminating presentation of a monoid $M$. Then, every homotopical
completion of $X$ is a coherent convergent presentation of $M$.
\end{thm}

\begin{example}[Klein bottle monoid]
\label{exa:klein-bottle-completion}We consider the Klein bottle
monoid $K^{+}$, as defined in~\cite[Subsection~I.3.2]{DDGKM}. It
has the following presentation:
\begin{equation}
\left\langle a,b\,\middle\vert\,bab=a\right\rangle .\label{eq:klein_presentation}
\end{equation}
The name comes from the fact that $K^{+}$ is the submonoid generated
by $a$ and $b$ of the fundamental group of the Klein bottle generated
by $a$ and $b$ subject to relation $bab=a$. Every element of~$K^{+}$
admits a unique expression of the form $a^{p}b^{q}$ for $p,q\geq0$
or $a^{p}b^{q}a$ for $p\geq0$ and $q\geq1$. That form is called
canonical. 

Let us apply a homotopical completion procedure to the presentation~\eqref{eq:klein_presentation}. We have the generating $1$-cells
$a$ and $b$, and a single generating $2$-cell $\alpha:bab\Rightarrow a$.
Let us adopt the following termination order: comparing the lengths
of words, then applying lexicographic order, induced by $a<b$, if
words have the same length. For instance, $b<aa<ab$. The only critical
branching is $\left\{ \alpha ab,ba\alpha\right\} $, with source $babab$.
The homotopical completion procedure adjoins the generating $2$-cell
$\beta:baa\Rightarrow aab$, and the ~generating $3$-cell $A$ for
coherence. The generating $2$-cell $\beta$ causes only one new critical
branching, namely $\left\{ \alpha aa,ba\beta\right\} $ with source
$babaa$, which is confluent, hence only the generating $3$-cell
$B$ is adjoined. Diagrammatically, the generating $3$-cells have
the shapes as follows:
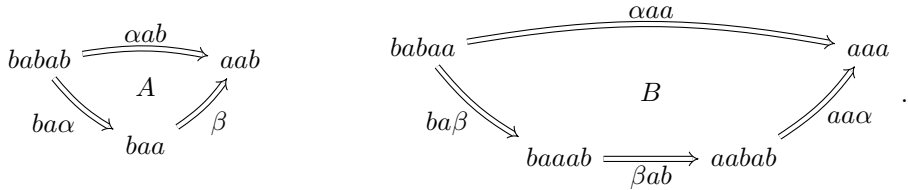
\[
\begin{tikzcd}[row sep=small, column sep=small] babab \arrow[Rightarrow, rr, bend left=10, "\alpha ab"{name=U}] \arrow[Rightarrow, rdd, bend right=10, "ba\alpha"'] & & aab \\ \\ & baa \arrow[from=U, phantom, "A"] \arrow[Rightarrow, ruu, bend right=10, "\beta"'] \end{tikzcd}
%\hskip \textwidth minus \textwidth 
\qquad\qquad
\begin{tikzcd}[row sep=scriptsize, column sep=scriptsize] babaa \arrow[Rightarrow, rrrr, bend left=10, "\alpha aa"{name=U}] \arrow[Rightarrow, rdd, bend right=10, "ba\beta"'] & & & & aaa \\ \\ & baaab \arrow[Rightarrow, rr, "\beta ab"'{name=D}] & & aabab \arrow[from=U, to=D, phantom, "B"] \arrow[Rightarrow, ruu, bend right=10, "aa\alpha"'] \end{tikzcd}.
\]
By Theorem~\ref{thm:homotopical-completion}, we have thus obtained
a convergent coherent presentation of the Klein bottle monoid, consisting
of two generating $1$-cells, two generating $2$-cells, and two generating
$3$-cells:
\[
\left(a,b\,\middle\vert\,bab\stackrel{\alpha}{\Rightarrow}a,baa\stackrel{\beta}{\Rightarrow}aab\,\middle\vert\,A,B\right).
\]
\end{example}

\begin{rem}
For convenience, we mostly leave implicit the orientation of the $3$-cells
in the diagrams. We only label the corresponding area with the name
of a $3$-cell. The convention is that the source and the target of
a $3$-cell are always the upper and the lower paths, respectively,
of the sphere bounding the area.
\end{rem}

\subsection{\label{subsec:reduction}Homotopical reduction}

A coherent presentation obtained by the homotopical completion procedure
is not necessarily minimal, in the sense that it may contain superfluous
cells. The homotopical reduction procedure aims to remove such superfluous
cells by performing a series of elementary collapses, analogous to
that used by Brown in~\cite{Bro}. We refer the reader to~\cite[Subsection~2.3]{GGM}
for a technical elaboration.

An \textbf{elementary Nielsen transformation} on a $\left(3,1\right)$-polygraph
$X$ is any of the following operations:
\begin{itemize}
\item replacement of a $2$-cell or a $3$-cell with its formal inverse;
\item replacement of a $3$-cell $A:\alpha\Rrightarrow\beta$ with \[
\begin{tikzcd}[row sep=scriptsize, column sep=large] & \ast \arrow[Rightarrow, r, shorten <=-2pt, shorten >=-2pt, "\alpha"{name=U}] & \ast \arrow[Rightarrow, rd, bend left=15, shorten <=-2pt, shorten >=-2pt, "\chi'"] & \\ \ast \arrow[Rightarrow, ru, bend left=15, shorten <=-2pt, shorten >=-2pt, "\chi"] \arrow[Rightarrow, rd, bend right=15, shorten <=-2pt, shorten >=-2pt, "\chi"'] & & & \ast \\ & \ast \arrow[Rightarrow, r, shorten <=-2pt, shorten >=-2pt, "\beta"'{name=D}] & \ast \arrow[Rightarrow, ru, bend right=15, shorten <=-2pt, shorten >=-2pt, "\chi'"'] \tarrow[from=U, to=D, shorten <=10pt, shorten >=10pt, "\widetilde{A}"]{} \end{tikzcd},
\]where $\chi$ and $\chi'$ are $2$-cells of $X_{3}^{\top}$.
\end{itemize}
Elementary Nielsen transformations preserve presented $1$-categories,
equivalence of presented $\left(2,1\right)$-categories and homotopy
type of $\left(3,1\right)$-polygraphs (see~\cite[2.1.4]{GGM}).
In particular, they transform a coherent presentation of a monoid
$M$ into another coherent presentation of $M$. A \textbf{Nielsen
transformation} is a composition of elementary ones. In a homotopical
completion-reduction procedure, Nielsen transformations are performed
implicitly for convenience.

Let $X$ be a $\left(3,1\right)$-polygraph. A generating $2$-cell (resp.\ $3$-cell,
resp.\ $3$-sphere) $\alpha$ of $X$ is called \textbf{collapsible}
if it meets the following two requirements:
\begin{itemize}
\item the target of $\alpha$ is a generating $1$-cell (resp.\ $2$-cell,
resp.\ $3$-cell) of $X$,
\item the source of $\alpha$ is a $1$-cell (resp.\ $2$-cell, resp.\ $3$-cell)
of the free $\left(3,1\right)$-category over $X\setminus\left\{ t\left(\alpha\right)\right\} $.
\end{itemize}
For a $\left(3,1\right)$-polygraph $X=\left(X_{0},X_{1},X_{2},X_{3}\right)$,
a \textbf{collapsible part} of $X$ is a triple $\Gamma=\left(\Gamma_{2},\Gamma_{3},\Gamma_{4}\right)$,
wherein $\Gamma_{2}$, $\Gamma_{3}$, $\Gamma_{4}$ respectively denote
families of generating $2$-cells of~$X$, generating $3$-cells of~$X$, $3$-spheres of $X_{3}^{\top}$, such that the following requirements
are met:
\begin{itemize}
\item every $\gamma$ of every $\Gamma_{k}$ is collapsible (possibly up
to a Nielsen transformation);
\item no $\gamma$ of $\Gamma_{k}$ is  the target of an element of $\Gamma_{k+1}$;
\item there exist well-founded order relations on $X_{1}$, $X_{2}$ and
$X_{3}$ such that, for every $\gamma$ in every~$\Gamma_{k}$, the
target of $\gamma$ is strictly greater than every generating $\left(k-1\right)$-cell
that occurs in the source of $\gamma$.
\end{itemize}
The result of the \textbf{homotopical reduction of $X$ with respect
to $\Gamma$} is the $\left(3,1\right)$-polygraph which we denote
$X/\Gamma$, whose generating cells are 
\[
X/\Gamma=\left(X_{0},X_{1}\setminus t\left(\Gamma_{2}\right),X_{2}\setminus t\left(\Gamma_{3}\right),X_{3}\setminus t\left(\Gamma_{4}\right)\right).
\]
Sources and targets are given by $\pi_{\Gamma}\circ s$ and $\pi_{\Gamma}\circ t$,
where $\pi_{\Gamma}$ is the $3$-functor from $X^{\top}$ to $\left(X/\Gamma\right)^{\top}$
given by the recursive formula
\[
\pi_{\Gamma}\left(x\right)=\begin{cases}
\pi_{\Gamma}\left(s\left(\gamma\right)\right) & \textrm{if }x=t\left(\gamma\right)\textrm{ for }\gamma\textrm{ in }\Gamma\\
1_{\pi_{\Gamma}\left(s\left(x\right)\right)} & \textrm{if }x\textrm{ in }\Gamma\\
x & \textrm{otherwise}.
\end{cases}
\]
Such a transformation is called the homotopical reduction procedure.

Let $X$ be a terminating $2$-polygraph, with a termination order
$\leq$. A \textbf{homotopical comple\-tion-reduction} of $X$ is a
$\left(3,1\right)$-polygraph, obtained as a homotopical reduction,
with respect to a collapsible part, of a homotopical completion of
$X$. Theorem~\ref{thm:squier} implies the following result.
\begin{thm}
\label{thm:completion-reduction}Assume that $X$ is a terminating
$2$-polygraph presenting a monoid $M$. Then, every homotopical completion-reduction
of $X$ is a coherent presentation of $M$.
\end{thm}

\subsection{\label{subsec:special-case}Special case of reduction}

We have just recalled the definition of a generic collapsible part
of a $\left(3,1\right)$-polygraph $X$. For the applications considered
here, however, it is practical to also recall a particular technique,
described in~\cite[3.2]{GGM}, to construct a collapsible
part in the case when $X$ is convergent and coherent. A \textbf{local
triple branching} is an unordered triple $\left\{ \alpha,\beta,\gamma\right\} $
of rewriting steps having a common source. A local triple branching
is \textbf{trivial} if two of its components are equal or if one of
its components forms branchings of the type $\left\{ \alpha v,u\beta\right\} $,
for $u=s\left(\alpha\right)$ and $v=s\left(\beta\right)$, with the
other two. In a manner analogous to the case of local branchings,
local triple branchings can be ordered by ``inclusion'', and a minimal
nontrivial local triple branching is called \textbf{critical}. A
\textbf{generating triple confluence} of $X$ is a particular kind
of $3$-sphere $\varPhi$ constructed using a critical triple branching.
Referring the reader to~\cite[Subsection~3.2]{GGM} for elaboration
of the technique, we illustrate it by means of an example.
\begin{example}
\label{exa:klein-bottle-reduction}Let us perform a homotopical reduction
procedure on the homotopical completion of the Klein bottle monoid,
computed in Example~\ref{exa:klein-bottle-completion}. We construct
a collapsible part $\Gamma=\left(\Gamma_{2},\Gamma_{3},\Gamma_{4}\right)$.
There is only one critical triple branching, namely $\left\{ \alpha abab,ba\alpha ab,baba\alpha\right\} $.
It yields a generating triple confluence, denoted $\varPhi$, whose
boundary consists of the following two parts (we display the $3$-cells
$A$ and $B$ differently now, to make the generating triple confluence
more evident):
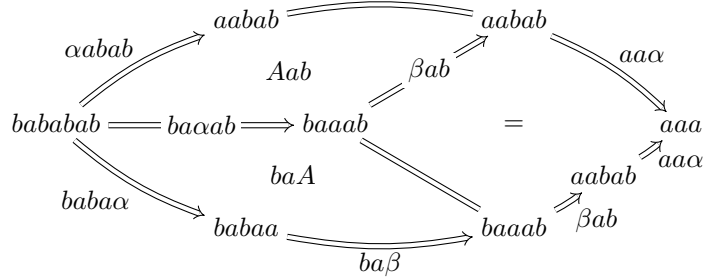
\[
\begin{tikzcd}[cramped, row sep=scriptsize, column sep=tiny, matrix scale=0.5, transform shape, nodes={scale=1}] & & aabab \arrow[equal, rrr, bend left=10] \arrow[rdd, phantom, "Aab"] & & & aabab \arrow[Rightarrow, rrdd, bend left=10, "aa\alpha"] & \\ & \phantom{aabab} \arrow[Rightarrow, ld, phantom] \arrow[Rightarrow, ru, phantom] & & & \phantom{aabab} \arrow[Rightarrow, ld, phantom] \arrow[Rightarrow, ru, phantom] \\ bababab \arrow[Rightarrow, rruu, bend left=10, "\alpha abab"] \arrow[Rightarrow, rrr, "ba\alpha ab" description] \arrow[Rightarrow, rrdd, bend right=10, "baba\alpha"'] & & & baaab \arrow[Rightarrow, rruu, "\beta ab" description] \arrow[equal, rrdd] \arrow[rrrr, phantom, "="] & & & & aaa \\ & & & & & & aabab \arrow[Rightarrow, ru, bend right=5, "aa\alpha"'] \\ & & babaa \arrow[ruu, phantom, "baA"] \arrow[Rightarrow, rrr, bend right=10, "ba\beta"'] & & & baaab \arrow[Rightarrow, ru, bend right=5, "\beta ab"'] \end{tikzcd}
\] and 
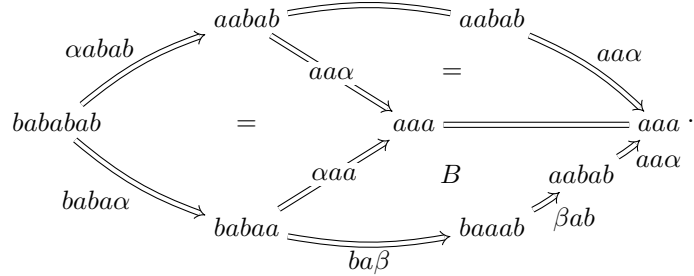
\[
\begin{tikzcd}[cramped, row sep=scriptsize, column sep=tiny, matrix scale=0.5, transform shape, nodes={scale=1}] & & aabab \arrow[equal, rrr, bend left=10] \arrow[Rightarrow, rrdd, "aa\alpha" description] & & & aabab \arrow[Rightarrow, rrdd, bend left=10, "aa\alpha"] & \\ & \phantom{aabab} \arrow[Rightarrow, ld, phantom] \arrow[Rightarrow, ru, phantom] \\ bababab \arrow[Rightarrow, rruu, bend left=10, "\alpha abab"] \arrow[rrrr, phantom, "="] \arrow[Rightarrow, rrdd, bend right=10, "baba\alpha"'] & & & & aaa \arrow[ruu, phantom, "="] \arrow[equal, rrr] \arrow[rdd, phantom, "B"] & & & aaa \\ & & & \phantom{aabab} \arrow[Rightarrow, ld, phantom] \arrow[Rightarrow, ru, phantom] & & & aabab \arrow[Rightarrow, ru, bend right=5, "aa\alpha"'] \\ & & babaa \arrow[Rightarrow, rruu, "\alpha aa" description] \arrow[Rightarrow, rrr, bend right=10, "ba\beta"'] & & & baaab \arrow[Rightarrow, ru, bend right=5, "\beta ab"'] &  \end{tikzcd}.
\]Hence the component $\Gamma_{4}$ of the collapsible part contains
the $3$-sphere $\varPhi$ which has the $3$-cell~$B$ as target
(recall that we implicitly perform a higher Nielsen transformation
when needed). Hence the component $\Gamma_{4}$ of the collapsible
part contains the $3$-sphere $\varPhi$ which has the $3$-cell $B$
as target. By the definition of a collapsible part, we also need to
provide a well-founded order relation on the set of generating $3$-cells,
such that, for every $3$-sphere $\left(X,Y\right)$ in $\Gamma_{4}$,
the target $Y$ is strictly greater than every generating $3$-cell
that occurs in the  source $X$. So, we put $B>A$. Proceeding as
described in~\cite[Subsection~3.2]{GGM}, we examine the remaining
$3$-cells and construct the component $\Gamma_{3}$ out of those
$3$-cells whose boundary contains a generating $2$-cell occurring
only once in the boundary. There is only one $3$-cell left, namely
$A$, and the $2$-cell $\beta$ appears only once in the boundary
of $A$. So, $\Gamma_{3}$ contains $A$, and we order the set of
generating $2$-cells by setting $\beta>\alpha$. The component $\Gamma_{2}$
is empty because there is no $2$-cell whose source or target consists
of a single generating $1$-cell appearing only once.

Thus, after performing a homotopical reduction procedure with respect
to the collapsible part $\left(\emptyset,\Gamma_{3},\Gamma_{4}\right)$,
we are left with the presentation
\[
\left(a,b\,\middle\vert\,bab\stackrel{\alpha}{\Rightarrow}a\,\middle\vert\,\emptyset\right)
\]
which is thus coherent by Theorem~\ref{thm:completion-reduction}.
Note that having a coherent presentation $X$ with the empty set of
generating $3$-cells means that any two parallel rewriting paths
represent the same $2$-cell in $X_{3}^{\top}$.
\end{example}

\subsection{\label{subsec:ggm-coherent-presentations}Application to Artin-Tits
and Garside monoids}

In this subsection, we recollect two instances of a homotopical completion-reduction
procedure, illustrating the results of~\cite[Section~3]{GGM}. We
shall recall these examples in Subsection~\ref{subsec:garside's-presentation-garside-family},
as the theorems of~\cite[Section~3]{GGM} are special cases of our
main result.

First, let us adopt a terminology concerning divisibility in monoids.
A monoid $M$ is \textbf{left-cancellative} (resp.\ \textbf{right-cancellative})
if for all $f$, $g$ and $g'$ of $M$, the equality $fg=fg'$ (resp.
$gf=g'f$) implies the equality $g=g'$. A monoid is \textbf{cancellative}
if it is both left-cancellative and right-cancellative.

An element $f$ of a monoid $M$ is said to be a \textbf{left divisor}
of $g\in M$, and $g$ is said to be a \textbf{right multiple} of
$f$, denoted by $f\preceq g$, if there is an element $f'\in M$
such that $ff'=g$. If, additionally, $f'$ is not invertible, then
divisibility  is called \textbf{proper}. We say that $f$ is a proper
left divisor of $g$, written as $f\prec g$, if $f\preceq g$ and
$g\npreceq f$. If $M$ is left-cancellative, then the element~$f'$
is uniquely defined and called \textbf{the right complement of $f$
in $g$}.

For an element $h$ of a left-cancellative monoid $M$ and a subfamily
$S$ of $M$, we say that $h$ is a \textbf{left-gcd} (resp.\ \textbf{right-lcm})
of $S$ if $h\preceq s$ (resp.\ $s\preceq h$) holds for all $s\in S$
and if every element of $M$ which is a left divisor (resp.\ right
multiple) of all $s\in S$ is also a left divisor (resp.\ right multiple)
of $h$.

A (proper) right divisor, a left multiple, a left complement, a left-lcm
and a right-gcd are defined similarly.

We say that a left-cancellative monoid $M$\textbf{ admits conditional
right-lcms} if any two elements having a common right multiple have
a right-lcm.
\begin{example}
\label{exa:artin-tits-coherent}Let $W$ be a Coxeter group (see e.g.
\cite[Section~3]{GGM}), and $B^{+}\left(W\right)$ the corresponding
Artin-Tits monoid. Garside's presentation of the $B^{+}\left(W\right)$,
seen as a $2$-polygraph and denoted by $\Gar_{2}\left(W\right)$,
has a single generating $0$-cell, elements of $W\setminus\left\{ 1\right\} $
as generating $1$-cells, and a generating $2$-cell
\[
\alpha_{u,v}:u|v\Rightarrow uv
\]
for all $u,v\in W\setminus\left\{ 1\right\} $ such that $\ell\left(uv\right)=\ell\left(u\right)+\ell\left(v\right)$
holds, where~$\ell(u)$ denotes the common length of all reduced expressions of~$u$. Let $\Gar_{3}\left(W\right)$ denote the extended presentation
of $B^{+}\left(W\right)$ obtained by adjoining to $\Gar_{2}\left(W\right)$
a generating $3$-cell\[
\begin{tikzcd}[row sep=scriptsize, column sep=scriptsize] & uv|w \arrow[Rightarrow, rd, bend left=10, "\alpha_{uv,w}"] \arrow[dd, phantom, "A_{u,v,w}"] & \\ u|v|w \arrow[Rightarrow, ru, bend left=10, "\alpha_{u,v}| w"] \arrow[Rightarrow, rd, bend right=10, "u|\alpha_{v,w}"'] & & uvw \\ & u|vw \arrow[Rightarrow, ru, bend right=10, "\alpha_{u,vw}"'] \end{tikzcd}
\]for all $u$, $v$ and $w$ of $W\setminus\left\{ 1\right\} $ such
that $\ell\left(uv\right)=\ell\left(u\right)+\ell\left(v\right)$
and $\ell\left(vw\right)=\ell\left(v\right)+\ell\left(w\right)$ and
$\ell\left(uvw\right)=\ell\left(u\right)+\ell\left(v\right)+\ell\left(w\right)$
hold. By~\cite[Theorem~3.1.3]{GGM}, $\Gar_{3}\left(W\right)$ is
a homotopical completion-reduction of $\Gar_{2}\left(W\right)$ so,
by Theorem~\ref{thm:completion-reduction}, it is a coherent presentation
of $B^{+}\left(W\right)$.
\end{example}

\begin{example}
\label{exa:garside-monoid-coherent}Recall that a \textbf{Garside
monoid} (see~\cite[Definition~I.2.1]{DDGKM}) is a pair $\left(M,\Delta\right)$
such that the following conditions hold:
\begin{enumerate}
\item $M$ is a cancellative monoid;
\item \label{enu:garside-monoid-2}there is a map $\lambda:M\to\mathbb{N}$
such that $\lambda\left(fg\right)\geq\lambda\left(f\right)+\lambda\left(g\right)$
and $\lambda\left(f\right)=0\implies f=1$;
\item every two elements have a left-gcd and a right-gcd and a left-lcm
and a right-lcm;
\item there is element $\Delta$, called the Garside element, such that
the left and the right divisors of $\Delta$ coincide, and they generate
$M$;
\item the family of all divisors of $\Delta$ is finite.
\end{enumerate}
We write $f\wedge g$ for the left-gcd of $f$ and $g$. For a (left)
divisor $f$ of $\Delta$, we write $\partial\left(f\right)$ for
the right complement of $f$ in $\Delta$.

Garside's presentation of a Garside monoid $M$ is the $2$-polygraph
$\Gar_{2}\left(M\right)$, having divisors of $\Delta$, other than
$1$, as generating $1$-cells and a generating $2$-cell $\alpha_{u,v}:u|v\Rightarrow uv$
whenever the condition $\partial\left(u\right)\wedge v=v$ is satisfied.
To be able to define generating $3$-cells, we need to generalise
this condition, in a suitable way, to three elements. Let us first
observe that the condition $\partial\left(u\right)\wedge v=v$ is
equivalent to saying that $v$ is a left divisor of $\partial\left(u\right)$.
In other words, there is~$w$ in~$M$ such that $vw=\partial\left(u\right)$.
By definition of $\partial\left(u\right)$, this means that $uvw=\Delta$,
so $uv$ is a divisor of $\Delta$. This reformulation allows an extension
of the given condition to a greater number of elements. Let $\Gar_{3}\left(M\right)$
denote the extended presentation of $M$ obtained by adjoining to
$\Gar_{2}\left(M\right)$ a generating $3$-cell\[
\begin{tikzcd}[row sep=scriptsize, column sep=scriptsize] & uv|w \arrow[Rightarrow, rd, bend left=10, "\alpha_{uv,w}"] \arrow[dd, phantom, "A_{u,v,w}"] & \\ u|v|w \arrow[Rightarrow, ru, bend left=10, "\alpha_{u,v}| w"] \arrow[Rightarrow, rd, bend right=10, "u|\alpha_{v,w}"'] & & uvw \\ & u|vw \arrow[Rightarrow, ru, bend right=10, "\alpha_{u,vw}"'] \end{tikzcd}
\]for all $u$, $v$ and $w$ divisors of $\Delta$, not equal to $1$,
such that $uv$, $vw$ and $uvw$ are divisors of $\Delta$. By Theorem~\cite[Theorem~3.3.3]{GGM}, $\Gar_{3}\left(M\right)$ is a homotopical
completion-reduction of $\Gar_{2}\left(M\right)$ so, by Theorem~\ref{thm:completion-reduction},
it is a coherent presentation of $M$.
\end{example}

\section{\label{sec:garside-families}Garside families}

This section briefly recollects the basic notions and results concerning
Garside families (for technical elaboration, see the book~\cite{DDGKM}).

\subsection{Right-mcms}

Let $M$ be a left-cancellative monoid, and $S$ a subfamily of $M$.
The left divisibility relation $\preceq$ is a preorder of elements;
it is an order if, and only if, $M$ has no nontrivial invertible
element.

A subfamily $S$ of a left-cancellative monoid $M$ is \textbf{closed
under right comultiple} if every common right multiple of two elements
$f$ and $g$ of $S$ (if there is any) is a right multiple of a common
right multiple of $f$ and $g$ that lies in $S$.

For $f$ and $g$ in a monoid $M$, a minimal common right multiple,
or \textbf{right-mcm}, of $f$ and $g$ if is a right multiple $h$
of $f$ and $g$, such that no proper left divisor of $h$ is a common
right multiple of $f$ and $g$. A monoid $M$ \textbf{admits right-mcms}
if, for all $f$ and $g$ of $M$, every common right multiple of
$f$ and $g$ is a right multiple of some right-mcm of $f$ and $g$.
Observe that in a monoid admitting conditional right-lcms, the notions
of a right-mcm and right-lcm coincide. Let us state a rather basic
observation about right-mcm in a left-cancellative monoid, which we
use in one step of the main proof in Subsection~\ref{subsec:main}.
The following lemma is similar to~\cite[Lemma~11.24]{HO}, which deals
with lcms whereas here it suffices to consider mcms (under weaker
assumptions).
\begin{lem}
\label{lem:left-multiple-of-right-mcm}Assume that $M$ is a left-cancellative
monoid. If $v'$ is a right-mcm of $v_{1}$ and $v_{2}$ in $M$,
then $uv'$ is a right-mcm of $uv_{1}$ and $uv_{2}$ for every $u$
in $M$.
\end{lem}

Following~\cite[Propositions~II.2.28 and II.2.29]{DDGKM}, a left-cancellative
monoid $M$ is said to be \textbf{left-noetherian} (resp.\ \textbf{right-noetherian})
if for every $g$ in $M$, every increasing sequence of right (resp.
left) divisors of $g$ with respect to proper right divisibility (resp.
left divisibility) is finite. A left-cancellative monoid $M$ is \textbf{noetherian}
if it is both left-noetherian and right-noetherian. 
\begin{example}
\uline{\label{exa:noetherian-artin-tits-garside}}Proper division,
left or right, strictly reduces the length of an element of an Artin-Tits
monoid. Therefore, no element admits an infinite number of divisors,
so Artin-Tits monoids are noetherian.

Garside monoids are noetherian by definition (thanks to the map $\lambda:M\to\mathbb{N}$).
\end{example}

\subsection{\label{subsec:garside-family}Notion of a Garside family}

In this subsection, we recollect the definition and some basic properties
of the all-important notion of a Garside family which provides a way
of extending the notion of a greedy decomposition beyond Garside monoids.

Given a subfamily $S$ of a left-cancellative monoid $M$, an $M$-word
$g_{1}|\cdots|g_{q}$ is said to be \textbf{$S$-greedy} if for all
$i<q$, 
\[
\forall h\in S,\forall f\in M,\left(h\preceq fg_{i}g_{i+1}\implies h\preceq fg_{i}\right).
\]
In other words, if the diagram 
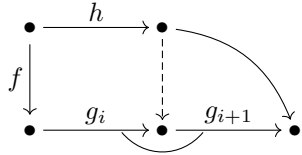
\[\begin{tikzcd}[row sep=large, column sep=large] \bullet \arrow[r, shorten <=-2pt,  shorten >=-2pt, "h"] \arrow[r, phantom, ""{name=g1'}', near end] \arrow[d, "f"'] & \bullet \arrow[rd, bend left, shorten <=-2pt,  shorten >=-2pt] \arrow[d, shorten <=-2pt, shorten >=-2pt, dashed] \arrow[d, phantom, ""{name=f1}', near start] & \\ \bullet \arrow[r, shorten <=-2pt,  shorten >=-2pt, "g_{i}"] \arrow[r, phantom, ""{name=g1}', near end] & \bullet \arrow[r, shorten <=-2pt,  shorten >=-2pt, "g_{i+1}"] \arrow[r, phantom, ""{name=g2}', near start] & \bullet \arrow[from=g1, to=g2, shorten <=-2pt,  shorten >=-2pt, bend right=50, dash] \end{tikzcd}\]commutes
without the dashed arrow, then there exists a dashed arrow making
the square on the left commute. The arc joining $g_{i}$ and $g_{i+1}$
denotes greediness. By definition, a word of length zero or one is
$S$-greedy for any subfamily $S$. 

Given a subfamily $S$ of $M$, an $M$-word $g_{1}|\cdots|g_{q}$
is said to be \textbf{$S$-normal} if it is \textbf{$S$}-greedy and
if, moreover, $g_{1},\ldots,g_{q}$ all lie in $S$. An $S$-normal
word $g_{1}|\cdots|g_{q}$ is \textbf{strict} if $g_{q}\neq1$. Observe
that the existence of an $S$-normal form implies the existence of
a strict one.

Note that, by the very definition of being greedy, a word is normal
if, and only if, its length-two factors are. More is true: the procedure
of transforming a word into its normal form consists of transforming
its length-two factors (we refer the reader to~\cite{DG} for elaboration).

In general, an \textbf{$S$-normal} decomposition of an element $g$
of $M$ is not unique. Nevertheless, the number of non-invertible
letters in all $S$-normal decompositions of $g$ is the same (see
\cite[Proposition~2.11]{DDM} or~\cite[Proposition~III.1.25]{DDGKM}
for exposition). If $M$ has no nontrivial invertible element, then
every $g$ in $M$ admits at most one strict $S$-normal decomposition.
Given a subfamily~$S$ of a left-cancellative monoid $M$, and an
element $g$ of $M$ admitting at least one $S$-normal decomposition,
one defines the \textbf{$S$-length} of an element $g\in M$ to be the common number of non-invertible
letters in all $S$-normal decompositions of $g$.

A subfamily $S$ of a left-cancellative monoid $M$ is called a \textbf{Garside
family} in $M$ if every element of $M$ admits an $S$-normal decomposition.
Since every left-cancellative monoid $M$ is a Garside family in itself
(for every $g$ in $M$, simply take a length-one word $g$ as a $M$-normal
decomposition of $g$), we are interested only in proper (meaning
other than $M$ itself) Garside families. Observe that, if $M$ has
no nontrivial invertible element and $S$ is a Garside family in~$M$,
then every element of $M$ admits a unique strict $S$-normal decomposition.
\begin{example}
Every Artin-Tits monoid admits a finite Garside family. In the case
of an Artin-Tits monoid of spherical type, a finite Garside family
is given by the corresponding Coxeter group. In the particular case
of a braid monoid, the family of all simple braids is a Garside family.

The Coxeter group $W$ which corresponds to a general Artin-Tits monoid
$B^{+}\left(W\right)$ is a possibly infinite Garside family, but
$B^{+}\left(W\right)$ admits a finite Garside family in any case
(see~\cite{DDH}).

Any Garside monoid $\left(M,\Delta\right)$ has a finite Garside family
given by the family of all divisors o~ $\Delta$ (see~\cite[Proposition~2.18]{DDM}
or~\cite[Proposition~III.1.43]{DDGKM}).
\end{example}

The following proposition gives a simple characterisation of a Garside
family. 
\begin{prop}[{\cite[Proposition~3.1]{DDM} or~\cite[Proposition~III.1.39]{DDGKM}}]
\label{prop:garside-family-iff}A subfamily $S$ of a monoid~$M$
containing no nontrivial invertible element is a Garside family if,
and only if, the following conjunction holds: $S$ generates $M$
and every element of $S^{2}$ admits an $S$-normal decomposition. 
\end{prop}

Let us recall another characterisation of Garside family, one direction
whereof we invoke in Subsection~\ref{subsec:main}. More characterisations
of Garside families can be found in~\cite[Subsection 3.2]{DDM} or
in~\cite[Subsection~IV.1.2]{DDGKM}. 
\begin{prop}[{\cite[Proposition~3.9]{DDM}}]
\label{prop:garside-implies-closed-comultiple}A family $S$ of a
left-cancellative monoid $M$ containing no nontrivial invertible
element is a Garside family if, and only if, the following conditions
are satisfied: $S$ generates $M$, it is closed under right comultiple
and right divisor, and every non-invertible element of $S^{2}$ admits
a $\prec$-maximal left divisor in $S$. 
\end{prop}

We recall another result to be used in Subsection~\ref{subsec:main}. 
\begin{lem}[{\cite[Lemma~IV.2.24]{DDGKM}}]
\label{lem:closed-comultiple-iff-closed-mcm}Assume that $M$ is
a left-cancellative monoid that contains no nontrivial invertible
element and admits right-mcms. Then for every subfamily $S$ of $M$,
the following are equivalent. 
\begin{itemize}
\item The family $S$ is closed under right comultiple. 
\item The family $S$ is \textbf{closed under right-mcm}, i.e.\ if $f$ and
$g$ lie in $S$, then so does every right-mcm of $f$ and $g$. 
\end{itemize}
\end{lem}

Given a Garside family $S$ in a left-cancellative monoid $M$ with
no nontrivial invertible element, the \textbf{normalisation map} $N^{S}:S^{\ast}\to S^{\ast}$
is the map which assigns to each $w\in S^{\ast}\setminus\left\{ 1\right\} $
the strict $S$-normal decomposition of the element of $M$ represented
by $w$; and $N^{S}\left(1\right)=1$. The following result provides
an important property of $S$-normal decomposition. 
\begin{lem}[{\cite[Lemma~6.9]{DG}}]
\label{lem:garside-implies-left-weighted}Assume that $M$ is a left-cancellative
monoid having no nontrivial invertible element, and $S$ is a Garside
family in $M$. For every word $w\in S^{\ast}$, the leftmost letter
of $w$ left-divides the leftmost letter of $N^{S}\left(w\right)$.
\end{lem}

\begin{proof}
Let $N^{S}\left(w\right)=s_{1}|\cdots|s_{q}$. Since $w$ and $s_{1}|s_{2}\cdots s_{q}$
evaluate to the same element of $M$, the leftmost letter of $w$
left-divides $s_{1}\left(s_{2}\cdots s_{q}\right)$. By \emph{\cite[Lemma~2.12]{DDM}},
for a subfamily $S$ of a left-cancellative monoid $M$, if an $M$-word
$g_{1}|\cdots|g_{q}$ is $S$-greedy, then $g_{1}|g_{2}\cdots g_{q}$
is $S$-greedy, as well.  Hence, the length-two $M$-word $s_{1}|s_{2}\cdots s_{q}$
is $S$-greedy.
\end{proof}
A normalisation map satisfying the conclusion of Lemma~\ref{lem:garside-implies-left-weighted},
but limited to $S$-words of length two, is called \textbf{left-weighted}
in~\cite[Subsection~6.2]{DG}, and Lemma~\ref{lem:garside-implies-left-weighted}
is a (quite straightforward) generalisation of~\cite[Lemma~6.9]{DG}
to all $S$-words.

A Garside family yields a presentation in the following sense. 
\begin{prop}[{\cite[Proposition~6.17]{DG} or~\cite[Corollary 6.6.4]{Gui}}]
\label{cor:dg-6.3.3}Assume that $M$ is a left-cancellative monoid
containing no nontrivial invertible element, and $S\subseteq M$ is
a Garside family. Then $M$ admits, as a convergent presentation,
the $2$-polygraph $\left(\left\{ \bullet\right\} ,S\setminus\left\{ 1\right\} ,X_{S}\right)$,
where $\left\{ \bullet\right\} $ denotes a singleton and $X_{S}$
is the set of generating $2$-cells of the form 
\begin{equation}
s|t\Rightarrow N^{S}\left(s|t\right)\label{eq:garside_quadratic_relations}
\end{equation}
for all $s$ and $t$ in $S\setminus\left\{ 1\right\} $ such that
$s|t$ is not $S$-normal. In particular, every Artin-Tits monoid
admits a finite convergent presentation. 
\end{prop}

A Garside family also induces a ``smaller'' presentation, beside
the one provided by Proposition~\ref{cor:dg-6.3.3}, which will be
instrumental in deriving our main result in the next section. The
following proposition is adapted from~\cite[Propositions~6.10 and 6.15]{DG}.
\begin{prop}
\label{prop:dg-6.3.1}Assume that $M$ is a left-cancellative monoid
containing no nontrivial invertible element, and $S\subseteq M$ is
a Garside family containing $1$. Then $M$ admits, as a presentation,
the $2$-polygraph $\Gar_{2}\left(S\right)$ which contains a single
generating $0$-cell, one generating $1$-cell for every element of
$S\setminus\left\{ 1\right\} $, and one generating $2$-cell of the
form
\begin{equation}
s|t\Rightarrow st,\label{eq:garside_triangular_relations}
\end{equation}
for all $s$ and $t$ in $S\setminus\left\{ 1\right\} $ whose product
$st$ in $M$ lies in $S$.
\end{prop}

\begin{proof}
Proposition~\ref{cor:dg-6.3.3} grants a presentation of $M$ in terms
of $S$ by the relations~\ref{eq:garside_quadratic_relations}. Let
us show that the relations~\eqref{eq:garside_triangular_relations}
are included in the relations~\eqref{eq:garside_quadratic_relations}.
If $s$ and $t$ in $S\setminus\left\{ 1\right\} $ are such that~$st$ lies in $S\setminus\left\{ 1\right\} $, then the strict $S$-decomposition
of $s|t$ is $st$, hence $N^{S}\left(s|t\right)=st$. Otherwise,
$st=1$ holds and yields $N^{S}\left(s|t\right)=1$. In both cases,
the strict $S$-normal decomposition of $s|t$ is~$st$. Hence, the
relations~\eqref{eq:garside_triangular_relations} are included in
the relations~\eqref{eq:garside_quadratic_relations}.

Conversely, let us show that each relation~\eqref{eq:garside_quadratic_relations},
with $s$ and $t$ in $S\setminus\left\{ 1\right\} $, follows from
a finite number of relations~\eqref{eq:garside_triangular_relations}.
Assume that $s$ and $t$ lie in $S\setminus\left\{ 1\right\} $ and
let $s'|t'\coloneqq N^{S}\left(s|t\right)$. If $t'=1$ holds, it
implies $s'=st$, which is a~\eqref{eq:garside_triangular_relations}
relation, so the result is true in this case. Otherwise, Lemma~\ref{lem:garside-implies-left-weighted}
implies that there exists $r$ in $M$, satisfying $sr=s'$, which
is a~\eqref{eq:garside_triangular_relations} relation. Being a right
divisor of $s'\in S$, the element $r$ also lies in $S$ by Proposition
\ref{prop:garside-implies-closed-comultiple}. Multiplying the equality
$sr=s'$ by $t'$ on the right yields $srt'=s't'=st$. Then the left
cancellation property of~$M$ implies $rt'=t$, which is a~\eqref{eq:garside_triangular_relations}
relation. Since the relation $s|t=s|r|t'=s'|t'$ follows from the~\eqref{eq:garside_triangular_relations} relations $s|r=s'$ and $r|t'=t$,
the result is true in this case, too.
\end{proof}
We call the $2$-polygraph $\Gar_{2}\left(S\right)$ the \textbf{Garside's
presentation} of $M$, with respect to the Garside family $S$. We
study it in the next section. Here, let us just observe that it extends
the Garside's presentation of Artin-Tits monoids, recalled in Subsection
\ref{subsec:ggm-coherent-presentations}.
\begin{example}
\label{exa:garside's-presentation-artin-tits-monoid}Garside's presentation
$\Gar_{2}\left(W\right)$ of an Artin-Tits monoid $B^{+}\left(W\right)$
is an instance of a Garside's presentation $\Gar_{2}\left(S\right)$
with respect to a Garside family $S$. Indeed, the Artin-Tits monoid
$B^{+}\left(W\right)$ is a cancellative monoid (see~\cite{BS}) with
no nontrivial invertible element, and the Coxeter group $W$ is a
Garside family containing $1$; hence $B^{+}\left(W\right)$ meets
all the requirements of Proposition~\ref{prop:dg-6.3.1} which, for
this particular input of $W$ for $S$, produces precisely Garside's
presentation $\Gar_{2}\left(W\right)$.
\end{example}

\begin{example}
\label{exa:garside's-presentation-garside-monoid}Garside's presentation
$\Gar_{2}\left(M\right)$ of a Garside monoid $M$ is another instance
of a Garside's presentation with respect to Garside family $S$. Namely,
$M$ is cancellative, by definition. Note that the property~\eqref{enu:garside-monoid-2}
of Garside monoid implies that it has no nontrivial invertible element.
All the divisors of $\Delta$ form a Garside family. If we take this
Garside family for $S$, Proposition~\ref{prop:dg-6.3.1} yields Garside's
presentation $\Gar_{2}\left(M\right)$.
\end{example}

\section{\label{sec:contribution}Coherent presentations from Garside families}

Having recalled necessary notions and results in previous sections,
in this section we aim to state and prove Theorem~\ref{thm:coherent-presentation-garside-family}
which provides a unifying generalisation of theorems recalled in Examples
\ref{exa:artin-tits-coherent} and~\ref{exa:garside-monoid-coherent}.

\subsection{\label{subsec:garside's-presentation-garside-family}Main statement
and sketch of proof}

In this subsection, we adapt some notation from~\cite{GGM} and set
a convenient noetherianity condition. Then we state our main result.

Let $M$ be a monoid generated by a set $S$ containing $1$. We
define the notations \typedeux{1} and \typedeux{0}, as follows.
Given two elements $u$ and $v$ of $S\setminus\left\{ 1\right\} $,
we write:
\begin{eqnarray*}
\typedeux{1} & \iff & uv\in S,\\
\typedeux{0} & \iff & uv\notin S.
\end{eqnarray*}
The notation extends to a greater number of elements. For three elements
$u,v,w\in S$, we write \typetrois{?} if both conditions $uv\in S$
and $vw\in S$ hold. The condition \typetrois{?} splits into two
mutually exclusive subcases:
\begin{eqnarray*}
\typetrois{1} & \iff & \begin{pmatrix}\typetrois{?} & \textrm{and} & uvw\in S\end{pmatrix},\\
\typetrois{0} & \iff & \begin{pmatrix}\typetrois{?} & \textrm{and} & uvw\notin S\end{pmatrix}.
\end{eqnarray*}

We formally redefine symbols $\Gar_{2}$ and $\Gar_{3}$ in our general
context as follows. The $2$-polygraph $\Gar_{2}\left(S\right)$ contains:
a single generating $0$-cell; one generating $1$-cell for every
element of $S\setminus\left\{ 1\right\} $; one generating $2$-cell
of the form
\[
\alpha_{u,v}:u|v\Rightarrow uv,
\]
for all $u$ and $v$ in $S\setminus\left\{ 1\right\} $ such that
\typedeux{1} holds. Here, $u|v$ denotes product in $S^{\ast}$,
whereas $uv$ denotes product in $M$. The $\left(3,1\right)$-polygraph
$\Gar_{3}\left(S\right)$ is consisting of the $2$-polygraph $\Gar_{2}\left(S\right)$
and the generating $3$-cells of the form \[
\begin{tikzcd}[row sep=scriptsize, column sep=scriptsize] & uv|w \arrow[Rightarrow, rd, bend left=10, "\alpha_{uv,w}"] \arrow[dd, phantom, "A_{u,v,w}"] & \\ u|v|w \arrow[Rightarrow, ru, bend left=10, "\alpha_{u,v}| w"] \arrow[Rightarrow, rd, bend right=10, "u|\alpha_{v,w}"'] & & uvw \\ & u|vw \arrow[Rightarrow, ru, bend right=10, "\alpha_{u,vw}"'] \end{tikzcd}
\]for all $u$, $v$ and $w$ in $S\setminus\left\{ 1\right\} $ such
that \typetrois{1}.
\begin{rem}
Note that the $2$-polygraph $\Gar_{2}\left(S\right)$ is not a presentation
of $M$, in general. Consequently, since $\Gar_{3}\left(S\right)$
is an extended presentation of a monoid presented by $\Gar_{2}\left(S\right)$,
it is not necessarily an extended presentation of $M$. Proposition
\ref{prop:dg-6.3.1} gives sufficient conditions for $\Gar_{2}\left(S\right)$
to be a presentation of $M$, thus making $\Gar_{3}\left(S\right)$
an extended presentation of $M$.

To formulate our main result, we need a restriction of right noetherianity
to a Garside family.
\end{rem}

\begin{defn}
Given a Garside family $S$ in a left-cancellative monoid $M$, we
say that $S$ is \textbf{right-noetherian} if for every $g$ in $S$,
every increasing sequence of proper left divisors in $S$ of $g$
with respect to proper left divisibility is finite.
\end{defn}

\begin{example}
\label{rem:ggm-noetherian}Every Garside family in a right-noetherian
left-cancellative monoid $M$ is right-noetherian.
\end{example}

Now, we state the main result.
\begin{thm}
\label{thm:coherent-presentation-garside-family}Assume that $M$
is a left-cancellative monoid containing no nontrivial invertible
element, and admitting a right-noetherian Garside family $S$ containing
$1$. If $M$ admits right-mcms, then $M$ admits the $\left(3,1\right)$-polygraph
$\Gar_{3}\left(S\right)$ as a coherent presentation.
\end{thm}

Before we proceed to prove the theorem, let us show that it gives
a common generalisation of the two distinct directions of extension,
given in~\cite{GGM}, of Deligne's result~\cite[Theorem~1.5]{Del}.
\begin{cor}[{\cite[Theorem~3.1.3]{GGM}}]
\label{thm:ggm-3.1.3-1}\label{rem:unifying-generalisation}For
every Coxeter group $W$, the Artin-Tits monoid $B^{+}\left(W\right)$
admits $\Gar_{3}\left(W\right)$ as a coherent presentation.
\end{cor}

\begin{proof}
Let us restrict the conditions \typedeux{1} and \typetrois{1}, defined
in the beginning of the current subsection, to the case of the Artin-Tits
monoid $B^{+}\left(W\right)$, with the Coxeter group $W$ as Garside
family $S$. Observe that, for $u,v\in W\setminus\left\{ 1\right\} $,
the condition \typedeux{1}, i.e.\ $uv\in W$, boils down to the condition
$\ell\left(uv\right)=\ell\left(u\right)+\ell\left(v\right)$ given
in Example~\ref{exa:artin-tits-coherent} (see Matsumoto's lemma,
e.g.~\cite[Corollary IX.1.11]{DDGKM}). Accordingly, the condition
\typetrois{1} becomes the conjunction of \typedeux{1} and \typedeuxbase{$v$}{$w$}{1}
and $\ell\left(uvw\right)=\ell\left(u\right)+\ell\left(v\right)+\ell\left(w\right)$.
Recall that Artin-Tits monoids and Garside monoids are cancellative
and noetherian (Example~\ref{exa:noetherian-artin-tits-garside}),
and that they contain no nontrivial invertible element. Consequently,
Theorem~\ref{thm:coherent-presentation-garside-family} specialises
to~\cite[Theorem~3.1.3]{GGM} when a monoid considered is Artin-Tits
with Coxeter group as a Garside family.
\end{proof}
Similarly, one shows that Theorem~\ref{thm:coherent-presentation-garside-family}
specialises to~\cite[Theorem~3.3.3]{GGM} when a monoid considered
is Garside with $S$ being the set of divisors of the Garside element.
\begin{cor}[{\cite[Theorem~3.3.3]{GGM}}]
\label{thm:ggm-3.3.3-1}Every Garside monoid $M$ admits $\Gar_{3}\left(M\right)$
as a coherent presentation.
\end{cor}

\begin{proof}
If we restrict the conditions \typedeux{1} and \typetrois{1} to
the case of a Garside monoid $\left(M,\Delta\right)$, with divisors
of $\Delta$ as Garside family $S$, then we get precisely our equivalent
reformulation, given in Example~\ref{exa:garside-monoid-coherent},
of the conditions stated in~\cite[Subsection~3.3]{GGM}. Literally,
the condition \typetrois{1} then says that $uv$ is an element of
the set of divisors of $\Delta$. Garside monoids are cancellative
by definition. Note that the property~\eqref{enu:garside-monoid-2}
of a Garside monoid implies noetherianity as well as the fact that
there are no nontrivial invertible elements.
\end{proof}
The following diagram summarises key steps of the proof and~\ref{exa:garside-monoid-coherent}
and thus motivates the next three subsections (which together contain
the proof).\[
\begin{tikzcd}[row sep=large, column sep=normal] \begin{array}{c}\textrm{Gar}_{3}(S)\\ \textrm{coherent, reduced} \end{array} & & \begin{array}{c}\textrm{\underline{Gar}}_{3}(S)\\ \textrm{coherent, convergent} \end{array} \arrow[ll, "\textrm{\normalsize homotopical}"', "\textrm{\normalsize reduction}"] \\ \\ \begin{array}{c}\textrm{Gar}_{2}(S)\\ \textrm{terminating} \end{array} \arrow[rr, "\textrm{\normalsize Knuth-Bendix}", "\textrm{\normalsize completion}"'] \arrow[uu, "\textrm{\normalsize homotopical}", "\textrm{\normalsize completion-reduction}" near start] & & \begin{array}{c}\textrm{\underline{Gar}}_{2}(S)\\ \textrm{convergent} \end{array} \arrow[uu, "\textrm{\normalsize Squier}"', "\textrm{\normalsize completion}"' near start] \end{tikzcd}
\]

In Subsection~\ref{subsec:termination}, starting with the Garside's
presentation $\Gar_{2}\left(S\right)$ of $M$, we add the generating
$2$-cells $\beta$ which results in a terminating presentation $\underline{\Gar}_{2}\left(S\right)$.
This is, in fact, a convergent presentation, namely a Knuth-Bendix
completion of $\Gar_{2}\left(S\right)$, but we do not prove it until
Subsection~\ref{subsec:main}. Nevertheless, this hindsight prompts
us to begin Subsection~\ref{subsec:termination} with a formal definition
of the $2$-polygraph $\underline{\Gar}_{2}\left(S\right)$.

In Subsection~\ref{subsec:main}, first we formally compute a Squier
completion of the polygraph $\underline{\Gar}_{2}\left(S\right)$,
under certain assumptions on the monoid. We denote the resulting $\left(3,1\right)$-polygraph
by $\underline{\Gar}_{3}\left(S\right)$. Then we show that this construction
applies to a terminating presentation $\underline{\Gar}_{2}\left(S\right)$
of $M$ and produces a coherent convergent presentation $\underline{\Gar}_{3}\left(S\right)$.

Finally, in Subsection~\ref{subsec:reduction-garside-family}, we
compute a homotopical reduction of $\underline{\Gar}_{3}\left(S\right)$
to obtain the $\left(3,1\right)$-polygraph $\Gar_{3}\left(S\right)$
as a coherent presentation of $M$.

\subsection{\label{subsec:termination}Attaining termination}

In this subsection, we ensure that a certain presentation, denoted
$\underline{\Gar}_{2}\left(S\right)$, is terminating. This presentation
arises naturally as a result of applying the Knuth-Bendix completion
to the Garside's presentation $\Gar_{2}\left(S\right)$.  Hence the
motivation for the formal definition of the $2$-polygraph $\underline{\Gar}_{2}\left(S\right)$.

Let $M$ be a monoid generated by a set $S$ containing $1$. Observe
that the $2$-polygraph $\Gar_{2}\left(S\right)$ has exactly one
critical branching for all $u$, $v$ and $w$ of $S\setminus\left\{ 1\right\} $
such that \typetrois{?} holds:

\[
\begin{tikzcd} uv|w & u|v|w \arrow[Rightarrow, l, "\alpha_{u,v}| w"'] \arrow[Rightarrow, r, "u|\alpha_{v,w}"] & u|vw. \end{tikzcd} 
\]If the subcase \typetrois{1} holds, then the branching is already
confluent. Otherwise \typetrois{0} holds, and the branching requires
a new generating $2$-cell to reach confluence, so the generating
$2$-cell $\beta_{u,v,w}:u|vw\Rightarrow uv|w$ is adjoined. We write
$\underline{\Gar}_{2}\left(S\right)$ for the $2$-polygraph which
contains a single generating $0$-cell, one generating $1$-cell for
every element of $S\setminus\left\{ 1\right\} $, the generating $2$-cells
\begin{align*}
\alpha_{u,v} & :u|v\Rightarrow uv, &  & u,v\in S\setminus\left\{ 1\right\} ,\quad\typedeux{1}\,,\\
\beta_{u,v,w} & :u|vw\Rightarrow uv|w, &  & u,v,w\in S\setminus\left\{ 1\right\} ,\quad\typetrois{0}\,.
\end{align*}

To show that the $2$-polygraph $\underline{\Gar}_{2}\left(S\right)$,
under certain conditions, is a Knuth-Bendix completion of the $2$-polygraph
$\Gar_{2}\left(S\right)$, we need to ensure two things: a way to
maintain a terminating presentation in the sense of Remark~\ref{rem:alternative-knuth-bendix},
and a demonstration that all new critical branchings caused by the
generating $2$-cells $\beta$ are confluent. These are respectively
given by Proposition~\ref{prop:termination}, and the proof of Proposition
\ref{prop:completion-assuming-termination}.

For an element $u$ of $S^{\ast}$, where $S$ is a set, we use the
following notations: $\ell\left(u\right)$ is the $S$-length of~$u$,
$\h\left(u\right)$ is the leftmost letter of~$u$, and~$\t\left(u\right)$ is
the word obtained by removing the letter~$\h\left(u\right)$ from~$u$.
\begin{prop}
\label{prop:termination}Assume that $M$ is a left-cancellative monoid
containing no nontrivial invertible element, admitting a right-noetherian
Garside family $S$ containing $1$. Then the $2$-polygraph $\underline{\Gar}_{2}\left(S\right)$
is terminating.
\end{prop}

\begin{proof}
Let us first adopt some notation. For a generating $2$-cell $\chi$,
a $\chi$-step is a rewriting step in which the generating $2$-cell
involved is $\chi$, and $\chi_{i}$ is a $\chi$-step
\[
\begin{tikzcd} \bullet \arrow[r, "w"] & \bullet \arrow[r, shorten <=-2pt,  shorten >=-2pt, bend left=50, "u"] \arrow[r, phantom, bend left=50, ""{name=U, below}] \arrow[r, shorten <=-2pt,  shorten >=-2pt, bend right=50, "v"'] \arrow[r, phantom, bend right=50, ""{name=D}] & \bullet \arrow[Rightarrow, from=U, to=D, shorten <=2pt, shorten >=2pt, "\chi"] \arrow[r, shorten <=-2pt,  shorten >=-2pt, "w'"] & \bullet \end{tikzcd},
\]where $w$ has length $i-1$. If $i_{1}|i_{2}|\cdots$ is an infinite
sequence of positive integers, we denote the path $\cdots\circ\chi_{i_{2}}\circ\chi_{i_{1}}$
by $\chi_{i_{1}|i_{2}|\cdots}$.

Suppose that there is an infinite rewriting path. Note that an $\alpha$-step
strictly reduces the $\left(S\setminus\left\{ 1\right\} \right)$-length
of a word, so there can be only finitely many of the generating $2$-cells
$\alpha$ in any rewriting path. Hence, there is no loss in generality
if we consider only $\beta$-steps. Namely, we can simply consider
an infinite path after the last $\alpha$-step is applied and we are
left with an infinite path containing only $\beta$-steps. So assume
that there is an infinite rewriting path of $\beta$-steps. Let $\beta_{i_{1}|i_{2}|\cdots}$
be such a path having source $u$ of minimal $\left(S\setminus\left\{ 1\right\} \right)$-length.
Note that $\ell\left(u\right)$ is at least two.

Note that the minimality assumption about $\ell\left(u\right)$ implies
that the position $1$ occurs infinitely many times in $i_{1}|i_{2}|\cdots$.
Namely, if the position $1$ occurred only finitely many times in
$i_{1}|i_{2}|\cdots$, then $\beta_{i_{k+1}-1|i_{k+2}-1|\cdots}$
would be an infinite path starting from $\t\left(\beta_{i_{1}|i_{2}|\cdots|i_{k}}\left(u\right)\right)$
of $\left(S\setminus\left\{ 1\right\} \right)$-length $\ell\left(u\right)-1$,
where $i_{k}=1$ is the last occurrence of $1$ in the sequence $i_{1}|i_{2}|\cdots$.
That would contradict the minimality assumption about $\ell\left(u\right)$.
We write $i_{c_{1}}|i_{c_{2}}|\cdots$ for the constant subsequence
of the sequence $i_{1}|i_{2}|\cdots$ taking all the members whose
value is $1$. In other words, $c_{1}$ is the least $j$ such that
$i_{j}=1$; and for all $n$, we have that $c_{n+1}$ is the least
$j$ such that conditions $j>c_{n}$ and $i_{j}=1$ hold.

Let $u^{\left(n\right)}$ denote the $n$th word in the path $\beta_{i_{1}|i_{2}|\cdots}$,
that is the source of the step $\beta_{i_{n}}$. Note that the leftmost
letter of the word is modified by a step $\beta_{i_{n}}$ if, and
only if, $i_{n}$ equals $1$. In this case, the modification is such
that the current leftmost letter $\h\left(u^{\left(n\right)}\right)$
is a proper left divisor of the next leftmost letter $\h\left(u^{\left(n+1\right)}\right)$,
and the corresponding complement lies in $S$ by the definition of
the generating $2$-cells $\beta$. In formal terms, 
\begin{equation}
\h\left(u^{\left(n+1\right)}\right)=\begin{cases}
\h\left(u^{\left(n\right)}\right) & \textrm{if }i_{n+1}\neq1,\\
\h\left(u^{\left(n\right)}\right)f_{n}\textrm{ for some }f_{n}\in S & \textrm{if }i_{n+1}=1.
\end{cases}\label{eq:leftmost-entry}
\end{equation}
Let $s$ denote the leftmost letter of the $S$-normal form of $u$.
Observe that all the words in the path $\beta_{i_{1}|i_{2}|\cdots}$
have the same evaluation in $M$ and that, consequently, the equality
$N^{S}\left(u\right)=N^{S}\left(u^{\left(n\right)}\right)$ holds
for all $n$ by the definition of $N^{S}$. By Lemma~\ref{lem:garside-implies-left-weighted},
we have that $\h\left(u^{\left(n\right)}\right)$ left-divides $s$
for all $n$.

Consider the sequence 
\begin{equation}
\left(\h\left(u^{\left(c_{n}\right)}\right)\right)_{n=1}^{\infty}\label{eq:infinite-sequence}
\end{equation}
of elements of $S$ that divide $g$. Observe that, by~\eqref{eq:leftmost-entry},
we have $\h\left(u^{\left(c_{n+1}\right)}\right)=\h\left(u^{\left(c_{n}\right)}\right)f_{c_{n}}$.
The existence of the sequence~\eqref{eq:infinite-sequence} contradicts
the fact that $S$ is right-noetherian. We conclude that the $2$-polygraph
$\underline{\Gar}_{2}\left(S\right)$ is terminating.
\end{proof}

\subsection{\label{subsec:main}Homotopical completion of Garside's presentation}

In this subsection, we enrich Garside's presentation to reach a coherent
convergent presentation. First (Proposition~\ref{prop:completion-assuming-termination})
we compute, purely formally, the homotopical completion of a terminating
presentation of a monoid satisfying certain conditions, but not presumed
to have a proper Garside family. Then we show, in Corollary~\ref{cor:coherent-convergent-presentation-garside-family},
that this provides a coherent convergent presentation of a left-cancellative
monoid containing no nontrivial invertible element, admitting right-mcms
and a right-noetherian Garside family containing $1$.
\begin{prop}
\label{prop:completion-assuming-termination}Assume that $M$ is a
left-cancellative monoid admitting right-mcms, and $S$ is a subfamily
of $M$ closed under right-mcm and right divisor. Assume that the
$2$-polygraph $\underline{\Gar}_{2}\left(S\right)$ is a terminating
presentation of $M$. Then $M$ admits, as a coherent convergent presentation,
the $\left(3,1\right)$-polygraph $\underline{\Gar}_{3}\left(S\right)$
which extends $\underline{\Gar}_{2}\left(S\right)$ with the following
twelve families of generating $3$-cells, indexed by all the possible
elements of $S\setminus\left\{ 1\right\} $:
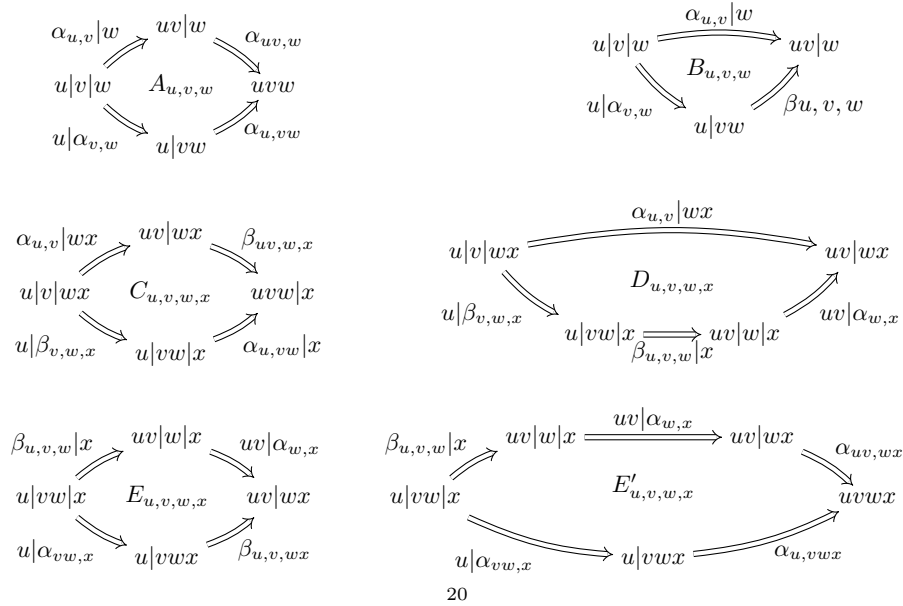
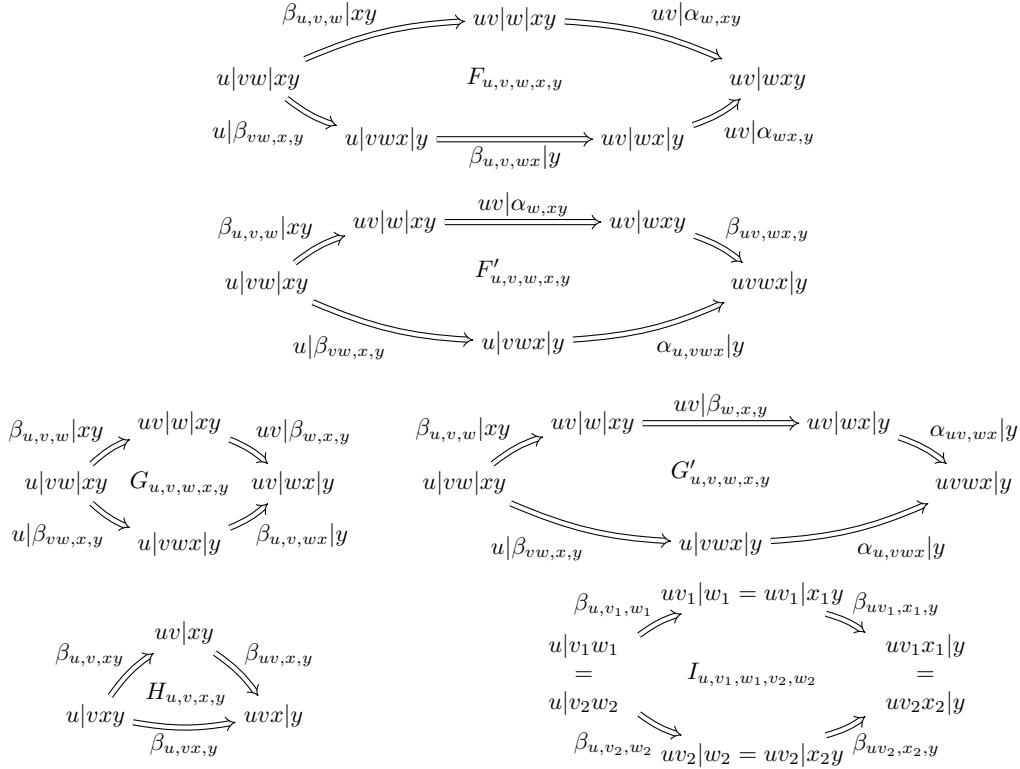
\[
\begin{tikzcd}[cramped, row sep=small, column sep=small, matrix scale=0.9, transform shape, nodes={scale=0.9}] & uv|w \arrow[Rightarrow, rd, bend left=10, "\alpha_{uv,w}"] \arrow[dd, phantom, "A_{u,v,w}"] & \\ u|v|w \arrow[Rightarrow, ru, bend left=10, "\alpha_{u,v}| w"] \arrow[Rightarrow, rd, bend right=10, "u|\alpha_{v,w}"'] & & uvw \\ & u|vw \arrow[Rightarrow, ru, bend right=10, "\alpha_{u,vw}"'] \end{tikzcd}
\qquad\qquad\qquad\qquad\qquad
\begin{tikzcd}[cramped, row sep=small, column sep=small, matrix scale=0.9, transform shape, nodes={scale=0.9}] u|v|w \arrow[Rightarrow, rr, bend left=10, "\alpha_{u,v}| w"{name=U}] \arrow[Rightarrow, rdd, bend right=10, "u|\alpha_{v,w}"'] & & uv|w \\ \\ & u|vw \arrow[from=U, phantom, "B_{u,v,w}"] \arrow[Rightarrow, ruu, bend right=10, "\beta{u,v,w}"'] \end{tikzcd}
\]

\[
\begin{tikzcd}[cramped, row sep=small, column sep=small, matrix scale=0.9, transform shape, nodes={scale=0.9}] & uv|wx \arrow[Rightarrow, rd, bend left=10, "\beta_{uv,w,x}"] \arrow[dd, phantom, "C_{u,v,w,x}"] & \\ u|v|wx \arrow[Rightarrow, ru, bend left=10, "\alpha_{u,v}|wx"] \arrow[Rightarrow, rd, bend right=10, "u|\beta_{v,w,x}"'] & & uvw|x \\ & u|vw|x \arrow[Rightarrow, ru, bend right=10, "\alpha_{u,vw}| x"'] \end{tikzcd}
\qquad\qquad
\begin{tikzcd}[cramped, row sep=small, column sep=small, matrix scale=0.9, transform shape, nodes={scale=0.9}] u|v|wx \arrow[Rightarrow, rrrr, bend left=10, "\alpha_{u,v}| wx"{name=U}] \arrow[Rightarrow, rdd, bend right=10, "u|\beta_{v,w,x}"'] & & & & uv|wx \\ \\ & u|vw|x \arrow[Rightarrow, rr, "\beta_{u,v,w}| x"'{name=D}] & & uv|w|x \arrow[from=U, to=D, phantom, "D_{u,v,w,x}"] \arrow[Rightarrow, ruu, bend right=10, "uv|\alpha_{w,x}"'] \end{tikzcd}
\]

\[
\begin{tikzcd}[cramped, row sep=small, column sep=small, matrix scale=0.9, transform shape, nodes={scale=0.9}] & uv|w|x \arrow[Rightarrow, rd, bend left=10, "uv|\alpha_{w,x}"] \arrow[dd, phantom, "E_{u,v,w,x}"] & \\ u|vw|x \arrow[Rightarrow, ru, bend left=10, "\beta_{u,v,w}| x"] \arrow[Rightarrow, rd, bend right=10, "u|\alpha_{vw,x}"'] & & uv|wx \\ & u|vwx \arrow[Rightarrow, ru, bend right=10, "\beta_{u,v,wx}"'] \end{tikzcd}
\qquad
\begin{tikzcd}[cramped, row sep=small, column sep=small, matrix scale=0.9, transform shape, nodes={scale=0.9}] & uv|w|x \arrow[Rightarrow, rr, "uv|\alpha_{w,x}"{name=U}] & & uv|wx \arrow[Rightarrow, rd, bend left=10, "\alpha_{uv,wx}"] \\ u|vw|x \arrow[Rightarrow, ru, bend left=10, "\beta_{u,v,w}| x"] \arrow[Rightarrow, rrd, bend right=10, "u|\alpha_{vw,x}"'] & & & & uvwx \\ & & u|vwx \arrow[from=U, phantom, "E'_{u,v,w,x}"] \arrow[Rightarrow, rru, bend right=10, "\alpha_{u,vwx}"'] \end{tikzcd}
\]

\[
\begin{tikzcd}[cramped, row sep=small, column sep=small, matrix scale=0.9, transform shape, nodes={scale=0.9}] & & uv|w|xy \arrow[to=D, phantom, "F_{u,v,w,x,y}"] \arrow[Rightarrow, rrd, bend left=10, "uv|\alpha_{w,xy}"] \\ u|vw|xy \arrow[Rightarrow, rru, bend left=10, "\beta_{u,v,w}| xy"] \arrow[Rightarrow, rd, bend right=10, "u|\beta_{vw,x,y}"'] & & & & uv|wxy \\ & u|vwx|y \arrow[Rightarrow, rr, "\beta_{u,v,wx}| y"'{name=D}] & & uv|wx|y \arrow[Rightarrow, ru, bend right=10, "uv|\alpha_{wx,y}"'] \end{tikzcd}
\]

\[
\begin{tikzcd}[cramped, row sep=small, column sep=small, matrix scale=0.9, transform shape, nodes={scale=0.9}] & uv|w|xy \arrow[Rightarrow, rr, "uv|\alpha_{w,xy}"{name=U}] & & uv|wxy \arrow[Rightarrow, rd, bend left=10, "\beta_{uv,wx,y}"] \\ u|vw|xy \arrow[Rightarrow, ru, bend left=10, "\beta_{u,v,w}| xy"] \arrow[Rightarrow, rrd, bend right=10, "u|\beta_{vw,x,y}"'] & & & & uvwx|y \\ & & u|vwx|y \arrow[from=U, phantom, "F'_{u,v,w,x,y}"] \arrow[Rightarrow, rru, bend right=10, "\alpha_{u,vwx}| y"'] \end{tikzcd}
\]

\[
\begin{tikzcd}[cramped, row sep=small, column sep=tiny, matrix scale=0.9, transform shape, nodes={scale=0.9}] & uv|w|xy \arrow[Rightarrow, rd, bend left=10, "uv|\beta_{w,x,y}"] \arrow[dd, phantom, "G_{u,v,w,x,y}"] & \\ u|vw|xy \arrow[Rightarrow, ru, bend left=10, "\beta_{u,v,w}| xy"] \arrow[Rightarrow, rd, bend right=10, "u|\beta_{vw,x,y}"'] & & uv|wx|y \\ & u|vwx|y \arrow[Rightarrow, ru, bend right=10, "\beta_{u,v,wx}| y"'] \end{tikzcd}
\qquad
\begin{tikzcd}[cramped, row sep=small, column sep=small, matrix scale=0.9, transform shape, nodes={scale=0.9}] & uv|w|xy \arrow[Rightarrow, rr, "uv|\beta_{w,x,y}"{name=U}] & & uv|wx|y \arrow[Rightarrow, rd, bend left=10, "\alpha_{uv,wx}| y"] \\ u|vw|xy \arrow[Rightarrow, ru, bend left=10, "\beta_{u,v,w}| xy"] \arrow[Rightarrow, rrd, bend right=10, "u|\beta_{vw,x,y}"'] & & & & uvwx|y \\ & & u|vwx|y \arrow[from=U, phantom, "G'_{u,v,w,x,y}"] \arrow[Rightarrow, rru, bend right=10, "\alpha_{u,vwx}| y"'] \end{tikzcd}
\]

\[
\begin{tikzcd}[cramped, row sep=small, column sep=tiny, matrix scale=0.9, transform shape, nodes={scale=0.9}] & uv|xy \arrow[to=D, phantom, "H_{u,v,x,y}"] \arrow[Rightarrow, rdd, bend left=10, "\beta_{uv,x,y}"] & \\ \\ u|vxy \arrow[Rightarrow, ruu, bend left=10, "\beta_{u,v,xy}"] \arrow[Rightarrow, rr, bend right=10, "\beta_{u,vx,y}"'{name=D}] & & uvx|y \end{tikzcd}
\qquad\qquad\qquad\qquad
\begin{tikzcd}[cramped, row sep=tiny, column sep=tiny, matrix scale=0.9, transform shape, nodes={scale=0.9}] & uv_1|w_1=uv_1|x_1y \arrow[Rightarrow, rd, bend left=10, "\beta_{uv_{1},x_{1},y}"] \arrow[dd, phantom, "I_{u,v_{1},w_{1},v_{2},w_{2}}"] & \\ \begin{array}{c} u|v_{1}w_{1}\\ =\\ u|v_{2}w_{2} \end{array} \arrow[Rightarrow, ru, bend left=10, "\beta_{u,v_{1},w_{1}}"] \arrow[Rightarrow, rd, bend right=10, "\beta_{u,v_{2},w_{2}}"'] & & \begin{array}{c} uv_{1}x_{1}|y\\ =\\ uv_{2}x_{2}|y \end{array} \\ & uv_2|w_2=uv_2|x_2y \arrow[Rightarrow, ru, bend right=10, "\beta_{uv_{2},x_{2},y}"'] \end{tikzcd}
\]
The meanings of the $1$-cells (i.e. words) $x_{1}$, $x_{2}$, $y$
and $x$, $y$ which appear respectively in the definitions of the
generating $3$-cells $I$ and $H$, are as follows. Since $v_{1}$
and $v_{2}$ have the common right multiple $v_{1}w_{1}=v_{2}w_{2}$,
they also have a right-mcm. The words $x_{1}$ and $x_{2}$ are the
right complements of $v_{1}$ and $v_{2}$, respectively, in their
right-mcm. The word $y$ is the right complement of $v_{1}x_{1}=v_{2}x_{2}$
in $v_{1}w_{1}=v_{2}w_{2}$. If either $x_{1}$ or $x_{2}$ is equal
to $1$, then the other one is simply denoted by $x$ (in the generating
$3$-cell $H$).

\medskip{}
\end{prop}

The structure of the following proof closely resembles that of the
proof of \cite[Proposition 3.2.1]{GGM}, but we need to devise more
general arguments to assure favourable properties in a more general
context.
\begin{proof}
Termination of the $2$-polygraph $\underline{\Gar}_{2}\left(S\right)$
is assumed, so we can perform a relaxed version of the Knuth-Bendix
completion procedure, as described in Remark \ref{rem:alternative-knuth-bendix},
simultaneously with the Squier completion procedure. It will turn
out that all critical branchings are confluent, and hence that only
a Squier completion will be actually computed, i.e. no further $2$-generating
cells will be added.

Let us first consider  critical branchings consisting only of the
generating $2$-cells $\alpha$. There is only one such critical branching
for all $u$, $v$ and $w$ of $S\setminus\left\{ 1\right\} $ such
that \typetrois{?} holds:

\[
\begin{tikzcd} uv|w & u|v|w \arrow[Rightarrow, l, "\alpha_{u,v}| w"'] \arrow[Rightarrow, r, "u|\alpha_{v,w}"] & u|vw. \end{tikzcd} 
\]If the subcase \typetrois{1} holds, the branching is already confluent,
so the homotopical completion procedure adjoins only the generating
$3$-cell $A_{u,v,w}$. If \typetrois{0} holds, the branching is
again confluent, so the generating $3$-cell $B_{u,v,w}$ is adjoined.

Let us now consider critical branchings containing the generating
$2$-cell $\beta$. The sources of $2$-cells forming such a branching
can either overlap on one element of $S\setminus\left\{ 1\right\} $
or be equal, as the lengths in $\left(S\setminus\left\{ 1\right\} \right)^{\ast}$
of the sources of the generating $2$-cells $\alpha$ and $\beta$
equal two. We consider the two cases accordingly.

For the first case, the proof of \cite[Proposition 3.2.1]{GGM} applies
here to a great extent. The source of a branching has length three,
as a word in $\left(W\setminus\left\{ 1\right\} \right)^{\ast}$.
One of the $2$-cells which form a branching, rewrites the leftmost
two generating $1$-cells of the source, and the other one rewrites
the rightmost two. There are three distinct forms of such branchings:\[\begin{gathered}
\begin{tikzcd}[column sep=large] uv|wx & u|v|wx \arrow[Rightarrow, l, "\alpha_{u,v}| wx"'] \arrow[Rightarrow, r, "u|\beta_{v,w,x}"] & u|vw|x, \end{tikzcd}
\\
\begin{tikzcd}[column sep=large] uv|w|x & u|vw|x \arrow[Rightarrow, l, "\beta_{u,v,w}| x"'] \arrow[Rightarrow, r, "u|\alpha_{vw,x}"] & u|vwx, \end{tikzcd}
\\
\begin{tikzcd}[column sep=large] uv|w|xy & u|vw|xy \arrow[Rightarrow, l, "\beta_{u,v,w}| xy"'] \arrow[Rightarrow, r, "u|\beta_{vw,x,y}"] & u|vwx|y. \end{tikzcd}
\end{gathered}
\]The first form is defined under the condition \typequatre{?}{0}{?},
which splits into two mutually exclusive possibilities \typequatre{1}{0}{?}
and \typequatre{0}{0}{?}, which respectively yield the generating
$3$-cells $C_{u,v,w,x}$ and $D_{u,v,w,x}$ by the homotopical completion
procedure. The second form is defined under the condition \typequatre{0}{1}{?}
which splits into \typequatre{0}{1}{0} and \typequatre{0}{1}{1},
which respectively produce the generating $3$-cells $E_{u,v,w,x}$
and $E'_{u,v,w,x}$. The third form is defined under the conditions
\typecinq{0}{1}{?}{?}{0}{?}. This situation splits into two mutually
exclusive possibilities \typecinq{0}{1}{1}{?}{0}{?} and \typecinq{0}{1}{0}{?}{0}{?}.
The former possibility further splits into \typecinq{0}{1}{1}{0}{0}{0}
and \typecinq{0}{1}{1}{1}{0}{0} which respectively yield the generating
$3$-cells $F_{u,v,w,x,y}$ and $F'_{u,v,w,x,y}$; the latter splits
into \typecinq{0}{1}{0}{0}{0}{0} and \typecinq{0}{1}{0}{1}{0}{0}
yielding the generating $3$-cells $G_{u,v,w,x,y}$ and $G'_{u,v,w,x,y}$,
respectively.

We have thus considered the first case. The second case is going to
be considered in greater detail because this is where new justifications
are needed. Assume that the two $2$-cells which generate a critical
branching, have the same source. One of those two $2$-cells has to
be a $\beta$ (otherwise, the branching is trivial). Therefore, the
source has to have a form $u|v_{1}w_{1}$ satisfying the condition
\typetroisbase{$u$}{$v_{1}$}{$w_{1}$}{0}. Since $2$-cells $\alpha$
are not defined under this condition, the other $2$-cell also has
to be a $\beta$. The only way for the generating $2$-cells $\beta$
with the same source $u|v_{1}w_{1}$ to form a critical branching
is for $v_{1}w_{1}$ to have another decomposition $v_{1}w_{1}=v_{2}w_{2}$
such that \typetroisbase{$u$}{$v_{2}$}{$w_{2}$}{0}. Then the branching
is as follows:\[
\begin{tikzcd}[column sep=large] uv_{1}|w_{1} & u|v_{1}w_{1}=u|v_{2}w_{2} \arrow[Rightarrow, l, "\beta_{u,v_{1},w_{1}}"'] \arrow[Rightarrow, r, "\beta_{u,v_{2},w_{2}}"] & uv_{2}|w_{2}. \end{tikzcd}
\]

Let us invoke the assumed property of $M$ admitting right-mcms.
Since $v_{1}$ and $v_{2}$ have a common right multiple, namely $v_{1}w_{1}=v_{2}w_{2}$,
they also have a right-mcm, say $v'$. Since $S$ is closed under
right-mcm by assumption, $v'$ lies in $S$. By the left cancellation
property which grants the uniqueness of right complements, we define
$x_{1}$ and $x_{2}$ as the right complements in $v'$ of $v_{1}$
and $v_{2}$, respectively. Since $S$ is closed under right divisor,
$x_{1}$ and $x_{2}$ are elements of $S$. We also define $y$ as
the right complement of $v'$ in $v_{1}w_{1}=v_{2}w_{2}$. Note that
$y$ is in $S$ as a right divisor of $v_{1}w_{1}$ which is in $S$.
Uniqueness of the right complements of $v_{1}$ and $v_{2}$ in $v_{1}w_{1}$
and $v_{2}w_{2}$, respectively, yields $w_{1}=x_{1}y$ and $w_{2}=x_{2}y$.
To sum up, the diagram 
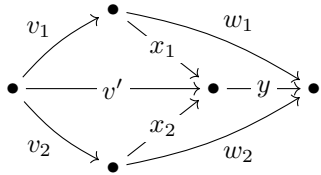
\[
\begin{tikzcd}[row sep=normal, column sep=normal] & \bullet \arrow[rrd, shorten <=-2pt,  shorten >=-2pt, bend left=10, "w_{1}"] \arrow[rd, shorten <=-2pt,  shorten >=-2pt, "x_{1}" description] & & \\ \bullet \arrow[ru, shorten <=-2pt,  shorten >=-2pt, bend left=10, "v_{1}"] \arrow[rr, shorten <=-2pt,  shorten >=-2pt, "v'" description] \arrow[rd, shorten <=-2pt,  shorten >=-2pt, bend right=10, "v_{2}"'] & & \bullet \arrow[r, shorten <=-2pt,  shorten >=-2pt, "y" description] & \bullet\\ & \bullet \arrow[ru, shorten <=-2pt,  shorten >=-2pt, "x_{2}" description] \arrow[rru, shorten <=-2pt,  shorten >=-2pt, bend right=10, "w_{2}"'] \end{tikzcd} 
\]commutes, where $\bullet$ denotes the unique $0$-cell. Furthermore,
the equality $w_{k}=x_{k}y$, the fact that $v'$ lies in $S$, and
the condition \typedeuxbase{$v_{k}$}{$w_{k}$}{1} together imply $\typetroisbase{\ensuremath{v_{k}}}{\ensuremath{x_{k}}}{\ensuremath{y}}{1}$
for $k\in\left\{ 1,2\right\} $.

We have only showed that $x_{1}$, $x_{2}$ and $y$ are elements
of $S$. Let us verify that all the generating $1$-cells involved
are, indeed, elements of $S\setminus\left\{ 1\right\} $. First we
demonstrate that $y$ cannot be equal to $1$. Assume the opposite.
Then the condition $\typetroisbase{\ensuremath{u}}{\ensuremath{v_{1}}}{\ensuremath{w_{1}}}{0}$
reduces to $\typetroisbase{\ensuremath{u}}{\ensuremath{v_{1}}}{\ensuremath{x_{1}}}{0}$.
On the other hand, $uv'$ is a right-mcm of $uv_{1}$ and $uv_{2}$
by Lemma \ref{lem:left-multiple-of-right-mcm}. Since $S$ is closed
under right-mcm, $uv'$ lies in $S$, which contradicts the condition
$\typetroisbase{\ensuremath{u}}{\ensuremath{v_{1}}}{\ensuremath{x_{1}}}{0}$.
Thus, we deduce that $y$ is not equal to $1$.

Note that if $x_{1}$ and $x_{2}$ were both equal to $1$, the branching
$\left\{ \beta_{u,v_{1},w_{1}},\beta_{u,v_{2},w_{2}}\right\} $ would
be trivial. So, at most one of the $1$-cells $x_{1}$ and $x_{2}$
can be equal to $1$. If $x_{2}=1$, the generating $3$-cell $H_{u,v,x,y}$
is constructed with $v\coloneqq v_{1}$ and $x\coloneqq x_{1}$. Similarly,
for $x_{1}=1$, the generating $3$-cell $H_{u,v,x,y}$ is constructed
with $v\coloneqq v_{2}$ and $x\coloneqq x_{2}$. Finally, if neither
$x_{1}$ nor $x_{2}$ is equal to $1$, the generating $3$-cell $I_{u,v_{1},w_{1},v_{2},w_{2}}$
is adjoined.

By Theorem \ref{thm:homotopical-completion}, the constructed $\left(3,1\right)$-polygraph
$\underline{\Gar}_{3}\left(S\right)$ is a coherent convergent presentation
of $M$.
\end{proof}
\begin{rem}
\uline{\label{rem:new-3-cells}}Observe that Proposition~\ref{prop:completion-assuming-termination}
gives three new families of generating $3$-cells (namely, $E'$,
$F'$ and $G'$) that were not a part of the \cite[Proposition 3.2.1]{GGM},
an analogous result for Artin-Tits monoids. The reason for this is
that the Garside families considered in \cite{GGM} for Artin-Tits
monoids and Garside monoids are closed under left and right divisors,
while a family $S$ in Proposition~\ref{prop:completion-assuming-termination}
is only closed under right divisor (like a Garside family in general).
Consequently, certain conjunctions of conditions, discussed in the
proof of Proposition~\ref{prop:completion-assuming-termination},
could not be satisfied in the setting of Artin-Tits monoids. For instance,
here we consider the possibility $uvwx\in S$ under the condition
$uvw\notin S$, among others, to construct the generating $3$-cell
$E'$. In an Artin-Tits monoid, on the other hand, $uvwx\in\sigma\left(W\right)$
would imply $uvw\in\sigma\left(W\right)$ due to closure under left
divisor.
\end{rem}
We can now deduce that the $2$-polygraph $\underline{\Gar}_{2}\left(S\right)$
is a Knuth-Bendix completion of the Garside's presentation $\Gar_{2}\left(S\right)$,
as hinted in Subsection~\ref{subsec:termination}.
\begin{cor}
\label{cor:convergent-presentation-garside-family}Assume that $M$
is a left-cancellative monoid containing no nontrivial invertible
element, and admitting a right-noetherian Garside family $S$ containing
$1$. Then the $2$-polygraph $\underline{\Gar}_{2}\left(S\right)$
is a convergent presentation of $M$.
\end{cor}

\begin{proof}
Proposition~\ref{prop:dg-6.3.1} grants that the $2$-polygraph $\Gar_{2}\left(S\right)$
is a presentation of $M$. Since the generating $2$-cells $\alpha$
strictly decrease the $S$-length, the $2$-polygraph $\Gar_{2}\left(S\right)$
is terminating. Thanks to Proposition~\ref{prop:termination}, we
can compute its Knuth-Bendix completion in a manner described in Remark
\ref{rem:alternative-knuth-bendix}. As shown in Subsection~\ref{subsec:termination},
the generating $2$-cells $\beta$ are added.

Note that Proposition~\ref{prop:garside-implies-closed-comultiple}
and Lemma~\ref{lem:closed-comultiple-iff-closed-mcm}, together with
the assumptions that $S$ contains $1$ and that $M$ contains no
nontrivial invertible element, yield the property of $S$ being closed
under right-mcm. By Proposition~\ref{prop:garside-implies-closed-comultiple},
$S$ is closed under right divisor. With all these conditions satisfied,
the proof of Proposition~\ref{prop:completion-assuming-termination}
applies in a straightforward fashion. In particular, it shows that
all new critical branchings caused by the generating $2$-cells $\beta$
are confluent. Thus, the $2$-polygraph $\underline{\Gar}_{2}\left(S\right)$
is a Knuth-Bendix completion of the Garside's presentation $\Gar_{2}\left(S\right)$,
which yields the desired conclusion by Theorem~\ref{thm:knuth-bendix}
and Remark~\ref{rem:alternative-knuth-bendix}.
\end{proof}
Observe that Proposition~\ref{prop:termination}, together with Proposition
\ref{prop:dg-6.3.1}, immediately implies that the $2$-polygraph
$\underline{\Gar}_{2}\left(S\right)$ is a terminating presentation
of $M$. On the other hand, the fact that $\underline{\Gar}_{2}\left(S\right)$
is also a convergent presentation of $M$ was reachable only after
Proposition~\ref{prop:completion-assuming-termination} when we made
sure that no additional generating $2$-cells were required to obtain
confluence.
\begin{cor}
\label{cor:coherent-convergent-presentation-garside-family}Assume
that $M$ is a left-cancellative monoid containing no nontrivial invertible
element, and admitting a right-noetherian Garside family $S$ containing
$1$. If $M$ admits right-mcms, then $M$ admits the $\left(3,1\right)$-polygraph
$\underline{\Gar}_{3}\left(S\right)$, defined in Proposition~\ref{prop:completion-assuming-termination},
as a coherent convergent presentation.
\end{cor}

\begin{proof}
Corollary~\ref{cor:convergent-presentation-garside-family} grants
that $\underline{\Gar}_{2}\left(S\right)$ is a terminating presentation
of $M$. As shown in the proof of Corollary~\ref{cor:convergent-presentation-garside-family},
all the requirements are met for applying Proposition~\ref{prop:completion-assuming-termination},
which completes the proof. 
\end{proof}

\subsection{\label{subsec:reduction-garside-family}Homotopical reduction of
Garside's presentation}

The homotopical reduction procedure from \cite[3.2.2]{GGM} applies
verbatim to the coherent convergent presentation provided by Proposition
\ref{prop:completion-assuming-termination} (and echoed by Corollary
\ref{cor:coherent-convergent-presentation-garside-family}), with
respect to a collapsible part $\Gamma$ obtained as follows. The component
$\Gamma_{4}$ of $\Gamma$ contains seven generating triple confluences
whose targets are the families $C$, ..., $I$ of generating $3$-cells,
with the order $I>H>\cdots>C$. For the sake of illustration, we recall
one such generating triple confluence in the case \typequatre{1}{1}{0}
(we refer the reader to \cite[3.2.2]{GGM} for the other six generating
triple confluences). Its boundary consists of the following two parts:
\[
\begin{tikzcd}[cramped, row sep=large, column sep=small, matrix scale=1, transform shape, nodes={scale=1}] & uv|w|x \arrow[Rightarrow, rr, bend left=10, "\alpha_{uv,w}|x"] \arrow[rd, phantom, "A_{u,v,w}|x"] & & uvw|x & \\ u|v|w|x \arrow[Rightarrow, ru, bend left=10, "\alpha_{u,v}|w|x"] \arrow[Rightarrow, rr, "u|\alpha_{v,w}|x" description] \arrow[Rightarrow, rd, bend right=10, "u|v|\alpha_{w,x}"'] & & u|vw|x \arrow[Rightarrow, ru, "\alpha_{u,vw}|x" description] \arrow[Rightarrow, rd, "u|\alpha_{vw,x}" description] \arrow[rr, phantom, "B_{u,vw,x}"] & & \null \\ & u|v|wx \arrow[ru, phantom, "u|A_{v,w,x}"] \arrow[Rightarrow, rr, bend right=10, "u|\alpha_{v,wx}"'] & & u|vwx \arrow[Rightarrow, uu, bend right=50, "\beta_{u,vw,x}"' description, near end] & \end{tikzcd} 
\]and \[
\begin{tikzcd}[cramped, row sep=large, column sep=small, matrix scale=1, transform shape, nodes={scale=1}] & uv|w|x \arrow[Rightarrow, rr, bend left=10, "\alpha_{uv,w}|x"{name=1-10}] \arrow[Rightarrow, rd, "uv|\alpha_{w,x}" description] & & uvw|x & \\ u|v|w|x \arrow[Rightarrow, ru, bend left=10, "\alpha_{u,v}|w|x"] \arrow[rr, phantom, "="] \arrow[Rightarrow, rd, bend right=10, "u|v|\alpha_{w,x}"'] & & uv|wx \arrow[Rightarrow, ru, "\beta_{uv,w,x}" description] \arrow[rr, phantom, "H_{u,v,w,x}"] \arrow[Leftarrow, rd, "\beta_{u,v,wx}" description] & & \null \\ & u|v|wx \arrow[Rightarrow, ru, "\alpha_{u,v}|wx" description] \arrow[Rightarrow, rr, bend right=10, "u|\alpha_{v,wx}"'{name=3-10}] & & u|vwx \arrow[Rightarrow, uu, bend right=50, "\beta_{u,vw,x}"' description, near end] \arrow[from=1-10, to=2-3, phantom, "B_{uv,w,x}" yshift=1.5ex] \arrow[from=3-10, to=2-3, phantom, "B_{u,v,wx}" yshift=-1.5ex] \end{tikzcd}.
\]The target of this particular generating triple confluence is the
generating $3$-cell $H_{u,v,w,x}$.

Note, however, that does not suffice to eliminate any of the generating
$3$-cells $E'_{u,v,w,x}$, $F'_{u,v,w,x,y}$ and $G'_{u,v,w,x,y}$
since these particular families of generating $3$-cells do not even
occur in \cite[Section 3]{GGM} (recall Remark~\ref{rem:new-3-cells}).
So, we have yet to eliminate these cells here. To this end, we consider
the following generating triple confluences in the $\left(3,1\right)$-polygraph
$\underline{\Gar}_{3}\left(S\right)$.

The boundary of our first $3$-sphere of interest consists of 
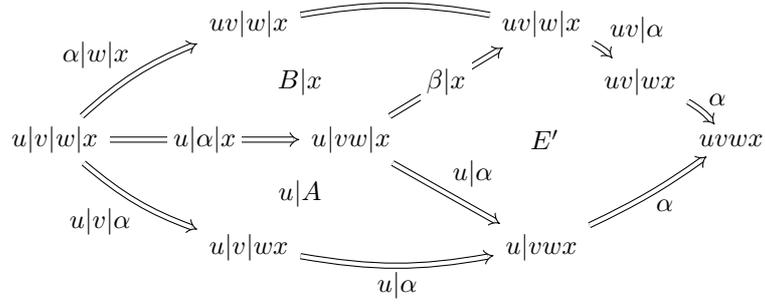
\[
\begin{tikzcd}[row sep=small, column sep=0em, matrix scale=0.5, transform shape, nodes={scale=1}] & & uv|w|x \arrow[equal, rrr, bend left=10] \arrow[rdd, phantom, "B|x"] & & & uv|w|x \arrow[Rightarrow, rd, bend left=10, "uv|\alpha"] & \\ & \phantom{aabab} \arrow[Rightarrow, ld, phantom] \arrow[Rightarrow, ru, phantom] & & & \phantom{aabab} \arrow[Rightarrow, ld, phantom] \arrow[Rightarrow, ru, phantom] & & uv|wx \arrow[Rightarrow, rd, bend left=10, "\alpha"] \\ u|v|w|x \arrow[Rightarrow, rruu, bend left=10, "\alpha|w|x"] \arrow[Rightarrow, rrr, "u|\alpha|x" description] \arrow[Rightarrow, rrdd, bend right=10, "u|v|\alpha"'] & & & u|vw|x \arrow[Rightarrow, rruu, "\beta|x" description] \arrow[Rightarrow, rrdd, "u|\alpha"] \arrow[rrrr, phantom, "E'"] & & & & uvwx \\ & & & & & & \phantom{aabab} \arrow[phantom, ru, bend right=5] \\ & & u|v|wx \arrow[ruu, phantom, "u|A"] \arrow[Rightarrow, rrr, bend right=10, "u|\alpha"'] & & & u|vwx \arrow[Rightarrow, rruu, bend right=5, "\alpha"'] \end{tikzcd} 
\]and
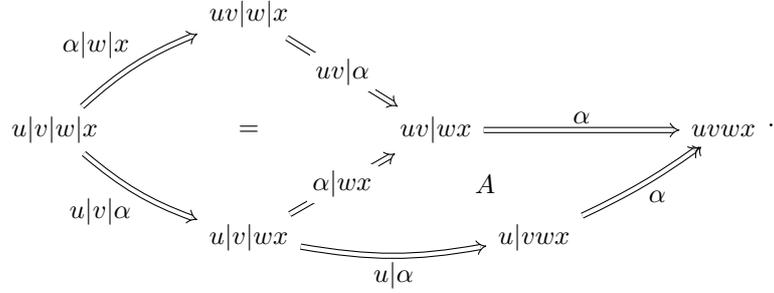
\[
\begin{tikzcd}[row sep=small, column sep=0em, matrix scale=0.5, transform shape, nodes={scale=1}] & & uv|w|x \arrow[phantom, rrr, bend left=10] \arrow[Rightarrow, rrdd, "uv|\alpha" description] & & & \phantom{uv|w|x} \arrow[phantom, rrdd, bend left=10, ""] & \\ & \phantom{aabab} \arrow[Rightarrow, ld, phantom] \arrow[Rightarrow, ru, phantom] & & & & & \phantom{uv|wx} \\ u|v|w|x \arrow[Rightarrow, rruu, bend left=10, "\alpha|w|x"] \arrow[rrrr, phantom, "="] \arrow[Rightarrow, rrdd, bend right=10, "u|v|\alpha"'] & & & & uv|wx \arrow[ruu, phantom, ""] \arrow[Rightarrow, rrr, "\alpha"] \arrow[rdd, phantom, "A"] & & & uvwx \\ & & & \phantom{aabab} \arrow[Rightarrow, ld, phantom] \arrow[Rightarrow, ru, phantom] & & & \phantom{aabab} \arrow[ru, bend right=5, ""'] \\ & & u|v|wx \arrow[Rightarrow, rruu, "\alpha|wx" description] \arrow[Rightarrow, rrr, bend right=10, "u|\alpha"'] & & & u|vwx \arrow[Rightarrow, rruu, bend right=5, "\alpha"'] &  \end{tikzcd}.
\]The target is the generating $3$-cell $E'_{u,v,w,x}$.

The second generating triple confluence which we are going to use has
the boundary consisting of
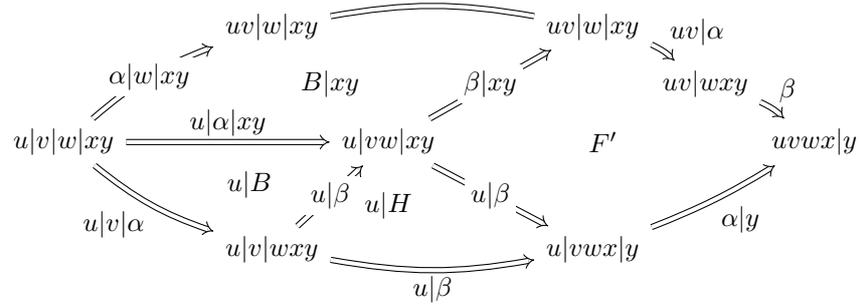
\[
\begin{tikzcd}[row sep=small, column sep=0em, matrix scale=0.5, transform shape, nodes={scale=1}] & & uv|w|xy \arrow[equal, rrr, bend left=10] \arrow[rdd, phantom, "B|xy"] & & & uv|w|xy \arrow[Rightarrow, rd, bend left=10, "uv|\alpha"] & \\ & \phantom{aabab} \arrow[Rightarrow, ld, phantom] \arrow[Rightarrow, ru, phantom] & & & \phantom{aabab} \arrow[Rightarrow, ld, phantom] \arrow[Rightarrow, ru, phantom] & & uv|wxy \arrow[Rightarrow, rd, bend left=10, "\beta"] \\ u|v|w|xy \arrow[Rightarrow, rruu, bend left=10, "\alpha|w|xy" description] \arrow[Rightarrow, rrr, "u|\alpha|xy"{name=U} description] \arrow[Rightarrow, rrdd, bend right=10, "u|v|\alpha"'] & & & u|vw|xy \arrow[dd, phantom, "u|H"] \arrow[Rightarrow, rruu, "\beta|xy" description] \arrow[Rightarrow, rrdd, "u|\beta" description] \arrow[rrrr, phantom, "F'"] & & & & uvwx|y \\ & & & & & & \phantom{aabab} \arrow[phantom, ru, bend right=5] \\ & & u|v|wxy \arrow[from=U, phantom, "u|B"] \arrow[ruu, Rightarrow, "u|\beta" description] \arrow[Rightarrow, rrr, bend right=10, "u|\beta"'] & \null & & u|vwx|y \arrow[Rightarrow, rruu, bend right=5, "\alpha|y"'] \end{tikzcd}
\]
and
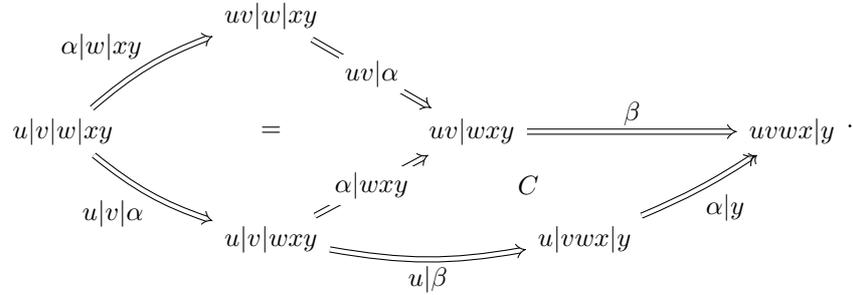
\[
\begin{tikzcd}[row sep=small, column sep=0em, matrix scale=0.5, transform shape, nodes={scale=1}] & & uv|w|xy \arrow[phantom, rrr, bend left=10] \arrow[Rightarrow, rrdd, "uv|\alpha" description] & & & \phantom{uv|w|x} \arrow[phantom, rrdd, bend left=10, ""] & \\ & \phantom{aabab} \arrow[Rightarrow, ld, phantom] \arrow[Rightarrow, ru, phantom] & & & & & \phantom{uv|wx} \\ u|v|w|xy \arrow[Rightarrow, rruu, bend left=10, "\alpha|w|xy"] \arrow[rrrr, phantom, "="] \arrow[Rightarrow, rrdd, bend right=10, "u|v|\alpha"'] & & & & uv|wxy \arrow[ruu, phantom, ""] \arrow[Rightarrow, rrr, "\beta"] \arrow[rdd, phantom, "C"] & & & uvwx|y \\ & & & \phantom{aabab} \arrow[Rightarrow, ld, phantom] \arrow[Rightarrow, ru, phantom] & & & \phantom{aabab} \arrow[ru, bend right=5, ""'] \\ & & u|v|wxy \arrow[Rightarrow, rruu, "\alpha|wxy" description] \arrow[Rightarrow, rrr, bend right=10, "u|\beta"'] & & & u|vwx|y \arrow[Rightarrow, rruu, bend right=5, "\alpha|y"'] &  \end{tikzcd}.
\]The target is the generating $3$-cell $F'_{u,v,w,x,y}$.

Finally, we construct the $3$-sphere whose boundary has the following
parts:
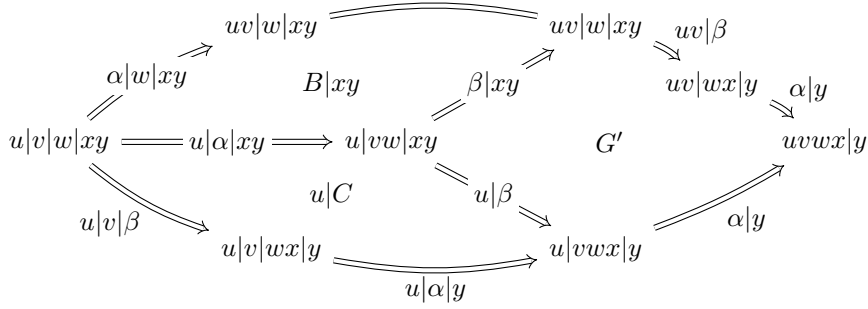
\[
\begin{tikzcd}[row sep=small, column sep=0em, matrix scale=0.5, transform shape, nodes={scale=1}] & & uv|w|xy \arrow[equal, rrr, bend left=10] \arrow[rdd, phantom, "B|xy"] & & & uv|w|xy \arrow[Rightarrow, rd, bend left=10, "uv|\beta"] & \\ & \phantom{aabab} \arrow[Rightarrow, ld, phantom] \arrow[Rightarrow, ru, phantom] & & & \phantom{aabab} \arrow[Rightarrow, ld, phantom] \arrow[Rightarrow, ru, phantom] & & uv|wx|y \arrow[Rightarrow, rd, bend left=10, "\alpha|y"] \\ u|v|w|xy \arrow[Rightarrow, rruu, bend left=10, "\alpha|w|xy" description] \arrow[Rightarrow, rrr, "u|\alpha|xy" description] \arrow[Rightarrow, rrdd, bend right=10, "u|v|\beta"'] & & & u|vw|xy \arrow[Rightarrow, rruu, "\beta|xy" description] \arrow[Rightarrow, rrdd, "u|\beta" description] \arrow[rrrr, phantom, "G'"] & & & & uvwx|y \\ & & & & & & \phantom{aabab} \arrow[phantom, ru, bend right=5] \\ & & u|v|wx|y \arrow[ruu, phantom, "u|C" description] \arrow[Rightarrow, rrr, bend right=10, "u|\alpha|y"'] & & & u|vwx|y \arrow[Rightarrow, rruu, bend right=5, "\alpha|y"'] \end{tikzcd}
\]

and
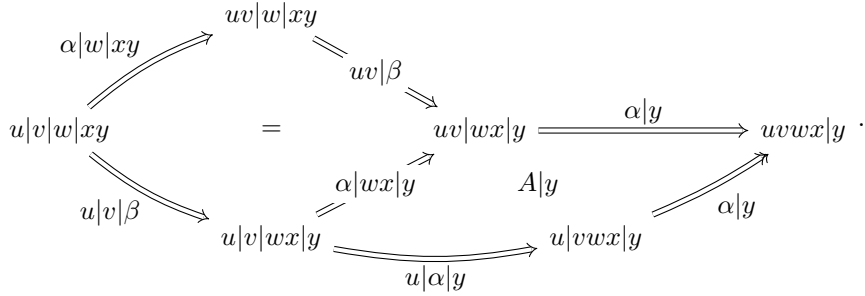
\[
\begin{tikzcd}[row sep=small, column sep=0em, matrix scale=0.5, transform shape, nodes={scale=1}] & & uv|w|xy \arrow[phantom, rrr, bend left=10] \arrow[Rightarrow, rrdd, "uv|\beta" description] & & & \phantom{uv|w|x} \arrow[phantom, rrdd, bend left=10, ""] & \\ & \phantom{aabab} \arrow[Rightarrow, ld, phantom] \arrow[Rightarrow, ru, phantom] & & & & & \phantom{uv|wx} \\ u|v|w|xy \arrow[Rightarrow, rruu, bend left=10, "\alpha|w|xy"] \arrow[rrrr, phantom, "="] \arrow[Rightarrow, rrdd, bend right=10, "u|v|\beta"'] & & & & uv|wx|y \arrow[ruu, phantom, ""] \arrow[Rightarrow, rrr, "\alpha|y"] \arrow[rdd, phantom, "A|y"] & & & uvwx|y \\ & & & \phantom{aabab} \arrow[Rightarrow, ld, phantom] \arrow[Rightarrow, ru, phantom] & & & \phantom{aabab} \arrow[ru, bend right=5, ""'] \\ & & u|v|wx|y \arrow[Rightarrow, rruu, "\alpha|wx|y" description] \arrow[Rightarrow, rrr, bend right=10, "u|\alpha|y"'] & & & u|vwx|y \arrow[Rightarrow, rruu, bend right=5, "\alpha|y"'] &  \end{tikzcd}.
\]The target is the generating $3$-cell $G'_{u,v,w,x,y}$.

So we extend the above mentioned component $\Gamma_{4}$ of the collapsible
part (inherited from \cite[3.2.2]{GGM}) with these three freshly
constructed $3$-spheres. We also extend the order relation on generating
$3$-cells to $G'>F'>E'>I>H>\cdots>C$. The component $\Gamma_{3}$
of the collapsible part contains the family $B$ of generating $3$-cells
having the generating $2$-cells $\beta$ as targets, with the order
$\beta>\alpha$.

The homotopical reduction of the resulting $\left(3,1\right)$-polygraph
of Proposition~\ref{prop:completion-assuming-termination}, with
respect to the collapsible part $\Gamma$, is precisely $\Gar_{3}\left(S\right)$.
By Theorem~\ref{thm:completion-reduction}, we conclude that $\Gar_{3}\left(S\right)$
is a coherent presentation of $M$. Through Corollary~\ref{cor:coherent-convergent-presentation-garside-family},
the proof of Theorem~\ref{thm:coherent-presentation-garside-family}
is hereby completed.

\subsection{Noetherianity}

Let us state an immediate corollary of Theorem~\ref{thm:coherent-presentation-garside-family},
having somewhat simpler (although more restrictive) requirements.
\begin{cor}
\label{cor:coherent-presentation-noetherian}Assume that $M$ is a
left-cancellative noetherian monoid containing no nontrivial invertible
element, and $S\subseteq M$ is a Garside family containing $1$.
Then $M$ admits the $\left(3,1\right)$-polygraph $\Gar_{3}\left(S\right)$
as a coherent presentation.
\end{cor}

\begin{proof}
Since $M$ is right-noetherian, so is $S$. By~\cite[Proposition~II.2.40]{DDGKM},
every left-cancellative left-noetherian monoid admits right-mcms,
so $M$ admits right-mcms. Hence, all the conditions of Theorem~\ref{thm:coherent-presentation-garside-family}
are satisfied.
\end{proof}
The next section demonstrates advantages of using our results in
applications. The following example, however, shows that taking a
Garside family as a generating set is not always the most practical
way to get a coherent presentation. 
\begin{example}
We revisit the Klein bottle monoid $K^{+}$ from (Examples~\ref{exa:klein-bottle-completion} and~
\ref{exa:klein-bottle-reduction}).
One of the infinitely many Garside families in $K^{+}$, none of which
is finite (see~\cite[Example~IV.2.35]{DDGKM}), is the set of left
divisors of $a^{2}$, which we denote by $S$. Let us check if the
conditions of Theorem~\ref{thm:coherent-presentation-garside-family}
are satisfied. Note that $K^{+}$ is cancellative as it is embeddable
in a group. The presentation (\ref{eq:klein_presentation}) contains
no relation of the form $u=v$ with exactly one the words $u$ and
$v$ being empty, hence $K^{+}$ has no nontrivial invertible element. Note that the left divisibility relation of $K^{+}$ is a
linear order (\cite[Figure I.6]{DDGKM}), which is a lot more than
necessary for admitting conditional right-lcms (consequently, right-mcms,
too). However, the sequence $\left(ab^{n}\right)_{n=1}^{\infty}$
shows that $S$ is not right-noetherian. Even worse, $S$ contains
an infinite path of the generating $2$-cells $\beta$, as defined
in Proposition~\ref{prop:termination}:
\[
b|a^{2}\mapsto b^{2}|aba\mapsto b^{3}|ab^{2}a\mapsto\cdots\mapsto b^{q}|ab^{q-1}a\mapsto\cdots
\]
Even if we took another Garside family, we would not be successful,
as witnessed by~\cite[Example~IV.2.35]{DDGKM}. Therefore, neither
Theorem~\ref{thm:coherent-presentation-garside-family} nor its proof
is applicable to $K^{+}$.

If one found a way to use a Garside family as a generating set, they
would have an infinite number of $1$-cells. On the other hand, by
directly performing the homotopical completion-reduction procedure
in Examples~\ref{exa:klein-bottle-completion} and~\ref{exa:klein-bottle-reduction},
we have demonstrated that the presentation~\eqref{eq:klein_presentation},
which has two generating $1$-cells and one generating $2$-cell,
is coherent. Therefore, for this particular example, the direct application
of the homotopical completion-reduction procedure is a preferable
way of reaching a coherent presentation.
\end{example}

\section{\label{sec:examples}Applications of Theorem~\ref{thm:coherent-presentation-garside-family}}

In this section, we consider applications of Theorem~\ref{thm:coherent-presentation-garside-family}
to certain monoids. In Subsections~\ref{subsec:abelian-infinite}
and~\ref{subsec:infinite-braids}, we apply it to monoids which are
neither Artin-Tits nor Garside. In Subsection~\ref{subsec:non-spherical},
we compute a finite coherent presentation of an Artin-Tits monoid
$B^{+}\left(W\right)$ that is not of spherical type, with a finite
Garside family $F$ (hence, $F\neq W$).

\subsection{\label{subsec:abelian-infinite}The free abelian monoid over an infinite
basis}

Consider the free abelian monoid~$\mathbb{N}^{\left(I\right)}$ of
all $I$-indexed sequences of nonnegative integers with finite support.
Note that $\mathbb{N}^{\left(I\right)}$ is not necessarily of finite
type, hence it is neither Artin-Tits nor Garside. Define
\[
S_{I}=\left\{ g\in\mathbb{N}^{\left(I\right)}\,\middle\vert\,\text{\ensuremath{\forall}}k\in I,g\left(k\right)\in\left\{ 0,1\right\} \right\} .
\]
Observe that $S_{I}$ is a Garside family in $\mathbb{N}^{\left(I\right)}$
(say, by applying Proposition~\ref{prop:garside-family-iff}). The
following properties follow from the fact that the definition of the
product on $\mathbb{N}^{\left(I\right)}$ is based on the pointwise
addition of nonnegative integers: $\mathbb{N}^{\left(I\right)}$ is
a cancellative monoid, it has no nontrivial invertible elements, and
it admits conditional right-lcms. Since every element of $\mathbb{N}^{\left(I\right)}$
has only finitely many divisors, $\mathbb{N}^{\left(I\right)}$ is
noetherian. So, all the conditions of Theorem~\ref{thm:coherent-presentation-garside-family}
are satisfied.

Let us describe the cells of the coherent presentation of $\mathbb{N}^{\left(I\right)}$
granted by Theorem~\ref{thm:coherent-presentation-garside-family}.
The generating $2$-cells are relations $\alpha_{u,v}:u|v\Rightarrow uv$
for $u,v\in S_{I}\setminus\left\{ 1\right\} $ with $uv\in S_{I}$,
which in this particular context means that $u$ and $v$ have disjoint
supports. A generating $3$-cell $A_{u,v,w}$ is adjoined for any
$u,v,w\in S_{I}\setminus\left\{ 1\right\} $ which have pairwise disjoint
supports.

As expected, for $I=\left\{ 1,2,\ldots,n\right\} $, we recover Garside's
presentation of the Artin-Tits monoid~$\mathbb{N}^{n}$ recalled in
Example~\ref{exa:artin-tits-coherent}, as well as Garside's presentation
of the Garside monoid $\mathbb{N}^{n}$ recalled in Example~\ref{exa:garside-monoid-coherent}.

\subsection{\label{subsec:infinite-braids}Infinite braids}

Denote by $B_{\infty}^{+}$ the monoid of all positive braids on
infinitely many strands indexed by positive integers, as defined in
\cite[Subsection~I.3.1]{DDGKM}. It is shown that $B_{\infty}^{+}$
is not of finite type, therefore it is neither Artin-Tits nor Garside.
Put
\[
S_{\infty}=\bigcup_{n\geq1}\left\{ \textrm{the family of all divisors of \ensuremath{\Delta_n}}\right\} ,
\]
where $\Delta_{n}$ denotes the half-turn braid on~$n$ strands. In other words, $S_{\infty}$
consists of all simple braids for all $n\geq1$. This is made precise
in~\cite[Subsection~I.3.1]{DDGKM}. Basically, $B_{n}^{+}$ is identified
with its image in $B_{n+1}^{+}$ under the homomorphism induced by
the identity map on $\left\{ \sigma_{1},\ldots,\sigma_{n}\right\} $.
In that sense, $B_{\infty}^{+}$ is seen as the union of all braid
monoids $B_{n}^{+}$. By Proposition~\ref{prop:garside-family-iff},
$S_{\infty}$ is a Garside family in $B_{\infty}^{+}$. Cancellation
properties, and having no nontrivial invertible elements are preserved
from braid monoids because the respective definitions do not depend
on $n$. The monoid is noetherian for the same reason as Artin-Tits
monoids (Example~\ref{exa:noetherian-artin-tits-garside}). So, we
can apply Theorem~\ref{thm:coherent-presentation-garside-family}
to construct a coherent presentation.

The generating $2$-cells are relations $\alpha_{u,v}:u|v\Rightarrow uv$
for $u,v\in S_{\infty}\setminus\left\{ 1\right\} $ whenever $uv\in S_{\infty}$,
which in this example means that $uv$ is a simple braid. A generating
$3$-cell $A_{u,v,w}$ is adjoined for any $u,v,w\in S_{\infty}\setminus\left\{ 1\right\} $
with $uv\in S_{\infty}$, $vw\in S_{\infty},$ and  $uvw\in S_{\infty}$,
which here means that~$uv$, $vw$ and~$uvw$ are simple braids. So,
formally, each cell is constructed exactly like in the coherent presentation
provided by~\cite[Theorem~3.1.3]{GGM} for a (finite) braid monoid,
regarded as an Artin-Tits monoid, which comes as no surprise because
Theorem~\ref{thm:coherent-presentation-garside-family} is a formal
generalisation of~\cite[Theorem~3.1.3]{GGM}.

\subsection{\label{subsec:non-spherical}Artin-Tits monoids that are not of spherical
type}

For an Artin-Tits monoid $B^{+}\left(W\right)$ of spherical type,
\cite[Theorem~3.1.3]{GGM} provides a finite coherent presentation
having $W\setminus\left\{ 1\right\} $ as a generating set. On the
other hand, if a Coxeter group $W$ is infinite, \cite[Theorem~3.1.3]{GGM}
still provides a coherent presentation but an infinite one. Recall
that every Artin-Tits monoid admits a finite Garside family (we refer
the reader to~\cite{DDH} for elaboration), regardless of whether
the monoid is of spherical type or not. An advantage of having Theorem
\ref{thm:coherent-presentation-garside-family} at our disposal is
that we can take a finite Garside family for a generating set in computing
a coherent presentation (whereas with~\cite[Theorem~3.1.3]{GGM},
one has to take the corresponding Coxeter group).

Let us consider the Artin-Tits monoid of type $\widetilde{A}_{2}$,
i.e.\ the monoid presented by
\begin{equation}
\left\langle \sigma_{1},\sigma_{2},\sigma_{3}\vert\sigma_{1}\sigma_{2}\sigma_{1}=\sigma_{2}\sigma_{1}\sigma_{2},\sigma_{2}\sigma_{3}\sigma_{2}=\sigma_{3}\sigma_{2}\sigma_{3},\sigma_{3}\sigma_{1}\sigma_{3}=\sigma_{1}\sigma_{3}\sigma_{1}\right\rangle ^{+}.\label{eq:non_spherical}
\end{equation}
By~\cite[Table 1 and Proposition~5.1]{DDH}, the smallest Garside
family $F$ in this monoid consists of the sixteen right divisors
of the elements $\sigma_{3}\sigma_{1}\sigma_{2}\sigma_{1}$, $\sigma_{1}\sigma_{2}\sigma_{3}\sigma_{2}$,
and $\sigma_{2}\sigma_{3}\sigma_{1}\sigma_{3}$. Namely, 
\begin{multline*}
F=\{1,\sigma_{1},\sigma_{2},\sigma_{3},\sigma_{1}\sigma_{2},\sigma_{2}\sigma_{1},\sigma_{2}\sigma_{3},\sigma_{3}\sigma_{2},\sigma_{3}\sigma_{1},\sigma_{1}\sigma_{3},\\
\sigma_{1}\sigma_{2}\sigma_{1},\sigma_{2}\sigma_{3}\sigma_{2},\sigma_{3}\sigma_{1}\sigma_{3},\sigma_{3}\sigma_{1}\sigma_{2}\sigma_{1},\sigma_{1}\sigma_{2}\sigma_{3}\sigma_{2},\sigma_{2}\sigma_{3}\sigma_{1}\sigma_{3}\}.
\end{multline*}
The Cayley graph of $F$ can be seen in~\cite[Figure 1]{DDH}.

As noted in Remark~\ref{rem:unifying-generalisation}, all the conditions
of Theorem~\ref{thm:coherent-presentation-garside-family} are satisfied.
Following Theorem~\ref{thm:coherent-presentation-garside-family},
we construct a generating $2$-cell $u|v\Rightarrow uv$ for $u,v\in F\setminus\left\{ 1\right\} $
with $uv\in F$. Thus we obtain three pairs of generating $2$-cells
of the form
\begin{align*}
\alpha_{\sigma_{i},\sigma_{j}} & :\sigma_{i}|\sigma_{j}\Rightarrow\sigma_{i}\sigma_{j} & \alpha_{\sigma_{j},\sigma_{i}} & :\sigma_{j}|\sigma_{i}\Rightarrow\sigma_{j}\sigma_{i},
\end{align*}
three pairs of generating $2$-cells of the form
\begin{align*}
\alpha_{\sigma_{i},\sigma_{j}\sigma_{i}} & :\sigma_{i}|\sigma_{j}\sigma_{i}\Rightarrow\sigma_{i}\sigma_{j}\sigma_{i} & \alpha_{\sigma_{j},\sigma_{i}\sigma_{j}} & :\sigma_{j}|\sigma_{i}\sigma_{j}\Rightarrow\sigma_{i}\sigma_{j}\sigma_{i},
\end{align*}
three pairs of generating $2$-cells of the form
\begin{align*}
\alpha_{\sigma_{i}\sigma_{j},\sigma_{i}} & :\sigma_{i}\sigma_{j}|\sigma_{i}\Rightarrow\sigma_{i}\sigma_{j}\sigma_{i} & \alpha_{\sigma_{j}\sigma_{i},\sigma_{j}} & :\sigma_{j}\sigma_{i}|\sigma_{j}\Rightarrow\sigma_{i}\sigma_{j}\sigma_{i},
\end{align*}
three generating $2$-cells of the form
\[
\alpha_{\sigma_{k},\sigma_{i}\sigma_{j}\sigma_{i}}:\sigma_{k}|\sigma_{i}\sigma_{j}\sigma_{i}\Rightarrow\sigma_{k}\sigma_{i}\sigma_{j}\sigma_{i},
\]
and three pairs of generating $2$-cells of the form
\begin{align*}
\alpha_{\sigma_{k}\sigma_{i},\sigma_{j}\sigma_{i}} & :\sigma_{k}\sigma_{i}|\sigma_{j}\sigma_{i}\Rightarrow\sigma_{k}\sigma_{i}\sigma_{j}\sigma_{i} & \alpha_{\sigma_{k}\sigma_{j},\sigma_{i}\sigma_{j}} & :\sigma_{k}\sigma_{j}|\sigma_{i}\sigma_{j}\Rightarrow\sigma_{k}\sigma_{i}\sigma_{j}\sigma_{i},
\end{align*}
with $i,j,k\in\left\{ 1,2,3\right\} $ and $j=i+1$ and $k=j+1$ modulo
$3$.

\medskip{}
We proceed to construct the generating $3$-cells $A_{u,v,w}$
for $u,v,w\in F\setminus\left\{ 1\right\} $ with $uv\in F$, $vw\in F,$
and $uvw\in F$. We obtain pairs of generating $3$-cells of the
form\[
\begin{tikzcd}[cramped, row sep=scriptsize, column sep=scriptsize, matrix scale=0.9, transform shape, nodes={scale=0.9}]
& \sigma_{i}\sigma_{j}|\sigma_{i} \arrow[Rightarrow, rd, bend left=10, "\alpha_{\sigma_{i}\sigma_{j},\sigma_{i}}"] \arrow[dd, phantom, "A_{\sigma_{i},\sigma_{j},\sigma_{i}}"] & \\
\sigma_{i}|\sigma_{j}|\sigma_{i} \arrow[Rightarrow, ru, bend left=10, "\alpha_{\sigma_{i},\sigma_{j}}|\sigma_{i}"] \arrow[Rightarrow, rd, bend right=10, "\sigma_{i}|\alpha_{\sigma_{j},\sigma_{i}}"'] & & \sigma_{i}\sigma_{j}\sigma_{i} \\
& \sigma_{i}|\sigma_{j}\sigma_{i} \arrow[Rightarrow, ru, bend right=10, "\alpha_{\sigma_{i},\sigma_{j}\sigma_{i}}"']
\end{tikzcd}
%\hskip \textwidth minus \textwidth 
\qquad\qquad
\begin{tikzcd}[cramped, row sep=scriptsize, column sep=scriptsize]
& \sigma_{j}\sigma_{i}|\sigma_{j} \arrow[Rightarrow, rd, bend left=10, "\alpha_{\sigma_{j}\sigma_{i},\sigma_{j}}"] \arrow[dd, phantom, "A_{\sigma_{j},\sigma_{i},\sigma_{j}}"] & \\
\sigma_{j}|\sigma_{i}|\sigma_{j} \arrow[Rightarrow, ru, bend left=10, "\alpha_{\sigma_{j},\sigma_{i}}|\sigma_{j}"] \arrow[Rightarrow, rd, bend right=10, "\sigma_{j}|\alpha_{\sigma_{i},\sigma_{j}}"'] & & \sigma_{i}\sigma_{j}\sigma_{i} \\
& \sigma_{j}|\sigma_{i}\sigma_{j} \arrow[Rightarrow, ru, bend right=10, "\alpha_{\sigma_{j},\sigma_{i}\sigma_{j}}"']
\end{tikzcd}
\]
or of the form
\[
\begin{tikzcd}[cramped, row sep=scriptsize, column sep=scriptsize, matrix scale=0.9, transform shape, nodes={scale=0.9}]
& \sigma_{k}\sigma_{i}|\sigma_{j}\sigma_{i} \arrow[Rightarrow, rd, bend left=10, "\alpha_{\sigma_{k}\sigma_{i},\sigma_{j}\sigma_{i}}"] \arrow[dd, phantom, "A_{\sigma_{k},\sigma_{i},\sigma_{j}\sigma_{i}}"] & \\
\sigma_{k}|\sigma_{i}|\sigma_{j}\sigma_{i} \arrow[Rightarrow, ru, bend left=10, "\alpha_{\sigma_{k},\sigma_{i}}|\sigma_{j}\sigma_{i}"] \arrow[Rightarrow, rd, bend right=10, "\sigma_{k}|\alpha_{\sigma_{i},\sigma_{j}\sigma_{i}}"'] & & \sigma_{k}\sigma_{i}\sigma_{j}\sigma_{i} \\
& \sigma_{k}|\sigma_{i}\sigma_{j}\sigma_{i} \arrow[Rightarrow, ru, bend right=10, "\alpha_{\sigma_{k},\sigma_{i}\sigma_{j}\sigma_{i}}"']
\end{tikzcd}
%\hskip \textwidth minus \textwidth 
\qquad\qquad
\begin{tikzcd}[cramped, row sep=scriptsize, column sep=scriptsize, matrix scale=0.9, transform shape, nodes={scale=0.9}]
& \sigma_{k}\sigma_{j}|\sigma_{i}\sigma_{j} \arrow[Rightarrow, rd, bend left=10, "\alpha_{\sigma_{k}\sigma_{j},\sigma_{i}\sigma_{j}}"] \arrow[dd, phantom, "A_{\sigma_{k},\sigma_{j},\sigma_{i}\sigma_{j}}"] & \\
\sigma_{k}|\sigma_{j}|\sigma_{i}\sigma_{j} \arrow[Rightarrow, ru, bend left=10, "\alpha_{\sigma_{k},\sigma_{j}}|\sigma_{i}\sigma_{j}"] \arrow[Rightarrow, rd, bend right=10, "\sigma_{k}|\alpha_{\sigma_{j},\sigma_{i}\sigma_{j}}"'] & & \sigma_{k}\sigma_{i}\sigma_{j}\sigma_{i} \\
& \sigma_{k}|\sigma_{i}\sigma_{j}\sigma_{i} \arrow[Rightarrow, ru, bend right=10, "\alpha_{\sigma_{k},\sigma_{i}\sigma_{j}\sigma_{i}}"']
\end{tikzcd}
\]
with $i$, $j$ and $k$ as above.

We have thus computed the finite coherent presentation of the Artin-Tits
monoid of type $\widetilde{A}_{2}$, which consists of fifteen generating
$1$-cells, twenty-seven generating $2$-cells, and twelve generating
$3$-cells.

Like in~\cite{GGM}, one can further perform a homotopical reduction
procedure. Here, the resulting $\left(3,1\right)$-polygraph contains:
a single generating $0$-cell; the generating $1$-cells $\sigma_{1}$,
$\sigma_{2}$, $\sigma_{3}$; the generating $2$-cells $\alpha_{\sigma_{2},\sigma_{1}\sigma_{2}}$,
$\alpha_{\sigma_{3},\sigma_{2}\sigma_{3}}$, $\alpha_{\sigma_{1},\sigma_{3}\sigma_{1}}$;
and no generating $3$-cells. As a side result, we have thus shown
that Artin's presentation of the Artin-Tits monoid of type $\widetilde{A}_{2}$,
with the empty set of generating $3$-cells, is coherent.

\printbibliography

\end{document}